\documentclass{article}
\usepackage[utf8]{inputenc}
\usepackage{tikz}
\usepackage{graphicx}
\usepackage{stmaryrd}
\usepackage{biblatex}
\usepackage{amsmath,amssymb}
 \usepackage{biblatex}
\usepackage{amsthm}
\usepackage{thmtools}
\usepackage{thm-restate}
\usepackage{amsfonts}
\usepackage{proof}
\usepackage{manfnt}
\usepackage{bbm}

\newtheorem{thm}{Theorem}

\newtheorem{lem}{Lemma}
\newtheorem{theorem}{Theorem}[section]
\newtheorem{proposition}[theorem]{Proposition}
\newtheorem{corollary}[theorem]{Corollary}

\newtheorem{lemma}[theorem]{Lemma}
\theoremstyle{definition}
\newtheorem{definition}[theorem]{Definition}

\newtheorem{remark}[theorem]{Remark}

\newcommand{\ZZ}{\mathbb{Z}}

\DeclareMathOperator{\Aut}{Aut}
\DeclareMathOperator{\Ker}{Ker}

\newcommand{\mgcon}{\sim_{M(G)}}
\title{Conjugacy in Miller's Groups}

\author{Conan Gillis}

\begin{document}
\newcommand{\numtheta}{\#_\theta}
\maketitle

\begin{abstract}
In 1971 C.F.\ Miller associated to every finitely presented group $G$ a free-by-free group $M(G)$ known as the  Miller Machine, whose conjugacy problem is closely related to the conjugacy and word problems of $G$. We quantify this relationship, and look to fully understand the conjugacy problem of $M(G)$; namely, we reduce the conjugacy problem in $M(G)$ to a strong form of list conjugacy in $G$, which we term iso-computational list conjugacy. As an application, we show that if $G$ is finite, the conjugacy problem for $M(G)$ is in $\mathsf{PSPACE}$.
\end{abstract}


\section{Introduction} 

Miller Machines, introduced in \cite[pg.\ 25]{Miller71}, give examples of residually finite groups with unsolvable conjugacy problem. Given a finite presentation of a group $G=\langle X\mid R\rangle$, the Miller Machine $M(G)$ is generated by the set $X$, a new set of letters $\Theta=\{\theta_\alpha:\alpha\in X\cup R\}$, and a further letter $q$, subject to the defining relations:
$$\begin{matrix}\theta_\alpha x&=&x\theta_\alpha&& x\in X,&\alpha\in X\cup R\\ \theta_xxq&=&qx\theta_x&&x\in X\ &\\ \theta_rq&=&qr\theta_r&&r\in R.\end{matrix}$$

Conjugation in $M(G)$ has an interesting relationship with the conjugacy and word problems of $G$: 
\begin{lem}[Miller \cite{Miller71}, Lemma III.A.4(2)] \label{millerlemma}    Let $u_1,u_2,v_1,v_2$ be words on $X^{\pm1}$. Then $u_1qu_2$ and $v_1qv_2$ are conjugate in $M(G)$ if and only if $u_1u_2$ and $v_1v_2$ are conjugate in $G$.
\end{lem}
An immediate consequence is

\begin{lem}[Sapir \cite{sapir_2011}]\label{Sapirlemma}
    For all words $u$ on $X^{\pm1}$, $qu$ and $q$ are conjugate if and only if~$u$~represents the identity of $G$.
\end{lem}
Borovik, Miasnikov, and Remeslennikov studied the conjugacy problem in $M(G)$ further in \cite{BMR}. They call a set $R\subseteq M(G)$ \textit{strongly negligible} if, letting $S_k$ be the sphere in $M(G)$ of radius $k$, $$\frac{|R\cap S_k|}{|S_k|}\leq \delta^k$$ for some fixed positive constant $\delta<1.$ They then prove the following theorem:
\begin{theorem}\label{BMRTheorem} The subgroup $\langle X,q\rangle_{M(G)}$ is strongly negligible. Moreover, there exists an algorithm which, given any elements $x,y$ of the set $M(G)\smallsetminus~\langle X,q\rangle_{M(G)}$, decides whether $x$ and $y$ are conjugate in $M(G)$. 
\end{theorem} 
This result shows that unsolvability of the conjugacy problem for $M(G)$ can only come from  $\langle X, q  \rangle$. In addition, the fact that $\langle X,q\rangle_{M(G)}$ is strongly negligible shows that, in a certain natural sense, the conjugacy problem is solvable for ``almost all" of $M(G)$. 

Here we complete the picture.  We find a decision problem in $G$, extending  Lemmas~\ref{millerlemma} and \ref{Sapirlemma}, which is equivalent to the conjugacy problem in $M(G)$ restricted to elements of $\langle X,q\rangle_{M(G)}$. Combining this with Theorem~\ref{BMRTheorem}, we have that our decision problem in $G$ is equivalent to the conjugacy problem in $M(G)$.


It is worth saying at the outset that the notation $M(G)$ depends on the presentation for $G$. Following \cite{sapir_2011}, however, we still use this notation with the understanding that we have fixed some finite presentation for $G$ beforehand. Our definitions and results below likewise assume a fixed presentation for $G$.

The statement of these results requires some additional notation; for precise definitions, see Section \ref{Preliminaries}. Firstly, if $x,y\in M(G)$ are conjugate via some $\gamma$ which can be written without $q$ as a factor, we write $x\approx y$. Secondly, let $\textbf{u}=~(u_1,...,u_k)$ and $\textbf{v}=(v_1,...,v_k)$ be $k$-tuples of elements of the free group $F(X)$ on $X$, and let $\sigma=(\sigma_1,...,\sigma_k)$~be a $k$-tuple of elements of $\{\pm1\}$. Throughout this paper, we adopt the notational convention that $\sigma_{k+1}=\sigma_1$, $u_{k+1}=u_k$, and $v_{k+1}=v_1$. 

With this convention in mind, we impose the condition that $\sigma_i=-\sigma_{i+1}$ implies that $u_i$ and $v_i$ are not both the identity of $F(X)$. We write $\textbf{u}\overset{\sigma}{\sim}\textbf{v}$ if there exists words $w,\varepsilon$ on $X$ such that $\varepsilon$ represents the identity of $G$ and $w$ and $w\varepsilon$ satisfy the following:

\begin{itemize}
    \item if $\sigma_i=\sigma_{i+1}=1$, then $w\varepsilon u_i w^{-1}= v_i$
    \item if $\sigma_i=1,\sigma_{i+1}=-1$, then $w\varepsilon u_i \varepsilon^{-1}w^{-1}= v_i$
    \item if $\sigma_{i}=-1,\sigma_{i+1}=1$, then $wu_iw^{-1}=v_i$
    \item if $\sigma_i=-1,\sigma_{i+1}=-1$, then $wu_i\varepsilon^{-1}w^{-1}= v_i$
\end{itemize} where all equalities between words are free. We call this relation \textit{iso-computational list conjugacy.} Note that, if $\sigma_i=-\sigma_{i+1}$ and both $u_i$ and $v_i$ are the trivial word, that the second and third bullet points immediately become vacuous. This, along with some technical details in the proofs below (see Theorem \ref{conjwithoutthetas} and its ancillary results), is the reason for our additional condition.

With this notation in hand, we can state our first main result.
 \begin{thm}
     The conjugcacy problem for $M(G)$ is decidable if and only if there is an algorithm deciding iso-computational list conjugacy in $G$.
 \end{thm}

 We will prove the reverse implication in Section \ref{mainalgorithm}. The forward implication will be postponed to Section \ref{forwardtheoremA}, after we have shown some useful quantitative results. To state these results, we need further notation.

For any word $w$ on a set $S$, let $||w||$ be the length of $w$ in the free group on $S$, under the usual word metric for $F(X)$. Additionally, $\langle X,q\rangle_{M(G)}$ and $\langle \Theta\rangle_{M(G)}$ are both (disjoint) free subgroups of $M(G)$, so we may define~$||w||$ for $w$ in $\langle X,q\rangle_{M(G)}$ or $\langle \Theta\rangle_{M(G)}$ to be the length of $w$ in $F(X\cup\{q\})$ or $F(\Theta)$ respectively. Note that $||w||$ will be at least the length of $w$ in~$M(G)$~itself, under the word metric given by the presentation described above. We denote this length by $|w|$. In the next two results, we bound conjugator length in $M(G)$ in terms of a close variant $\Lambda$ of the Dehn function of $G$, as well as a set of functions $\textbf{C}_{k,\sigma}(n)$. For fixed $k>0$ and $\sigma=(\sigma_1,...,\sigma_k)$, with $\sigma_i\in \{\pm1\}$, the latter functions measure the maximal value of $||w||$ over all $k$-tuples $\textbf{u},\textbf{v}$ such that $\textbf{u}\overset{\sigma}{\sim}\textbf{v}$ and $\sum_i ||u_i||+||v_i||\leq n$, where $w$ is as in the definition of $\textbf{u}\overset{\sigma}{\sim}\textbf{v}$. For full details on these functions, see Subsections \ref{lamlength} and \ref{ICLCdef} respectively. 

 Our first quantitative result of this part describes conjugator length in the case of Lemma~\ref{Sapirlemma}:
\begin{thm}
Let $D_0(n)$ be the restriction of the conjugator length function of the Miller Machine $M(G)$ to pairs $(q, qu)$ such that $q$ and $qu$ are conjugate in~$M(G)$, $u\in \langle X\rangle_{M(G)}$, and $||u||\leq n-2$. Then $D_0(n)$ is  within a constant factor of  $\Lambda(n-2)$.  More precisely, if $t=~\max(\{||r||:r\in R\}\cup\{2\})$, then for all $n>2$,  $$\frac{\Lambda(n-2)}{3t} \leq  D_0(n) \leq \Lambda(n-2).$$ 
\end{thm}

In the same vein, our next result may be viewed as a partial quantification of Theorem A.
\begin{thm} Fix $k>0$ and $\sigma=(\sigma_1,...,\sigma_k)$, where $\sigma_i\in \{\pm1\}$. Let $D_{k,\sigma}$ be the restriction of the conjugator length function of the Miller Machine $M(G)$ to pairs $x=q^{\sigma_1}u_1q^{\sigma_2}u_2 \cdots q^{\sigma_k}u_k$ and $y=~q^{\sigma_1}v_1q^{\sigma_2}v_2 \cdots q^{\sigma_k}v_k$ such that $x$ and $y$ are conjugate in~$M(G)$, $u_i,v_i\in \langle X\rangle_{M(G)}$ and $\sigma_i\in \{\pm1\}$, for all $i$, and $\sum_i \left( ||u_i||+||v_i||\right)\leq n-2k$. Also, let $t=~\max(\{||r||:~r\in R\}\cup \{2\})$. Then $$D_{k,\sigma}(n)\geq \frac{\textbf{C}_{k,\sigma}(n-2k)}{t}$$ for all $n$ large enough. Also, if there exists some $i$ such that $\sigma_i=\sigma_{i+1}$, then $$D_{k,\sigma}(n)\leq (2M+1)\Lambda(\textbf{C}_{k,\sigma}(n)+n)$$ for a constant $M$ depending on the presentation chosen for $G$. \end{thm}

Note that in the special case of $k=1$ and $\sigma=(1)$, this result quantifies Lemma \ref{millerlemma} in the case where $u_1$ and $v_1$ are trivial. In addition, the $\Lambda$ function can be used to give coarser bounds in terms of just $\textbf{C}_{k,\sigma}$ and the Dehn function of $G$ (see Proposition \ref{lambdadehnrel}). It is possible to extend these bounds to the conjugator length function defined on all of $M(G)$, however this is in many cases uncomputable, so we omit it here.

\paragraph{Acknowledgements}The author would like to thank A. Beaupre, R. Kleinberg, J. Manning, and T. Riley for their very helpful discussions during this work. The author is also grateful for the support of the National Science Foundation, under NSF Grant DGE - 2139899 and a question posed to him by A. Miasnikov.

\section{The Algebraic Structure of $M(G)$}\label{structureMG} 

\subsection{Notation}
For the entirety of this paper, let $G$ be a finitely presented group, with a fixed finite presentation $G=\langle X|R\rangle.$ For technical reasons, we assume without loss of generality that $R\neq \emptyset$ and $X\cap R=\emptyset.$ If two words $w_1$ and $w_2$ are equal in any group $H$ (typically $H=G$ or $M(G)$), we will write $w_1=_H w_2$. When $w_1$ and $w_2$ are freely equal, we will write $w_1\equiv w_2$, and when the context is clear we will simply write~$w_1=w_2$. When we are considering functions, such as $\varphi:g_1\mapsto g_2,$ we will just write $\varphi(g_1)=g_2$, with the equality understood as being in the codomain of the function.

Additionally, recall the definition of Miller Machines:
\begin{definition} Fix a fresh letter $\theta_\alpha$ for each $\alpha\in X\cup R,$ along with a fresh letter $q,$ and let $\Theta=\{\theta_\alpha|
\alpha\in X\cup R\}$. The Miller Machine $M(G)$ of $G$ is the group generated by $X\cup\Theta\cup\{q\},$ subject to the relations $$\begin{matrix}\theta_\alpha x&=&x\theta_\alpha&& x\in X,&\alpha\in X\cup R\\ \theta_xxq&=&qx\theta_x&&x\in X\ &\\ \theta_rq&=&qr\theta_r&&r\in R.\end{matrix}$$
Besides this presentation, \cite{Miller71} gives several other useful ways to view $M(G)$, which we discuss here. The arguments we give are due \cite{Miller71}, however we make explicit some details that will be useful later.
\end{definition}
\subsection{$M(G)$ as HNN-extension with stable letters $\theta_\alpha$}\label{stableletterthetaHNN} Let $\ \alpha\in X\cup R,$ and define $\phi_\alpha:X\cup\{q\}\to \langle X,q\rangle_{M(G)}$ by

\[\phi_\alpha(\chi)= \begin{cases} 
      \chi & \chi\in X \\
      \alpha^{-1}q\alpha& \alpha \in X\text{ and }\chi=q \\
      q\alpha& \alpha \in R\text{ and }\chi=q.
   \end{cases}
\]

By inspection of the relations of $M(G)$, $\langle X,q\rangle_{M(G)}$ is a free group, and we have defined $\phi_\alpha$ on all of the generators of $\langle X,q\rangle_{M(G)}$. Thus, $\phi_\alpha$ can be extended to a group homomorphism $\phi_\alpha:\langle X,q\rangle_{M(G)}\to \langle X,q\rangle_{M(G)}$ for all $\alpha\in X\cup R$. A well-known theorem of Nielsen \cite{Nielsen1924} implies that each $\phi_\alpha$ is an automorphism of $\langle X,q\rangle_{M(G)}$, so we can take the HNN extension of $\langle X,q\rangle_{M(G)}$ with along one of the $\phi_\alpha$'s, taking~$\theta_\alpha$~as the stable letter. This group will still have $\langle X,q\rangle_{M(G)}$ as a subgroup, so we can take another HNN extension along $\phi_{\alpha'}$ for any $\alpha'\neq~\alpha$. Doing this for every $\alpha'\neq \alpha$, in arbitrary order, will give a group $M'(G)$ generated by $X\cup \Theta\cup \{q\} $ and subject to relations $$\begin{matrix}\theta_\alpha x\theta_\alpha^{-1}&=&x&& x\in X,&\alpha\in X\cup R\\ \theta_xq\theta_x^{-1}&=&x^{-1}qx&&x\in X\ \\\theta_rq\theta_r^{-1}&=&qr&&r\in R.\end{matrix}$$ In particular, we see that $\theta_\alpha x=x\theta_\alpha$ holds in both groups, for all~$\alpha~\in ~X~\cup ~R$~. Using this relation we see that the second relations of both groups are equivalent: $$\theta_xq\theta_x^{-1}=x^{-1}qx\iff \theta_xq=x^{-1}qx\theta_x\iff x\theta_xq=qx\theta_x\iff \theta_xxq=qx\theta_x.$$ The third relations of both groups are equivalent as well: $$\theta_rq=qr\theta_r\iff \theta_r q \theta_r^{-1}=qr.$$ Since $M(G)$ and $M'(G)$ have the same generating sets, and all of $M(G)$'s relations are derivable in $M'(G)$ and vice versa, they are isomorphic groups.

\begin{remark} Note that $\theta_\alpha\mapsto \phi_\alpha$ gives a map $\langle \Theta\rangle_{M(G)}\to \Aut(\langle X,q\rangle_{M(G)}).$ This map is easily verified to be a homomorphism, so it gives a semidirect product structure $\langle X,q\rangle_{M(G)}\rtimes \langle \Theta\rangle_{M(G)}$ for $M(G)$. As a consequence, every element $x\in M(G)$ can be written uniquely as $\alpha\tau$ for some $\alpha\in \langle X,q\rangle_{M(G)}$ and $\tau\in~ \langle \Theta\rangle_{M(G)}.$
\end{remark}
\subsection{$M(G)$ as HNN-extension with stable letter $q$}\label{stabletterqHNN} Let $H=\langle X,\Theta\rangle_{M(G)}$, and define two subgroups of $H$: the subgroup $K_{-1}$ generated by $\{\theta_xx|x\in X\}\cup\{\theta_r|r\in R\}$, and the subgroup $K_1$ generated by $\{\theta_xx|x\in X\}\cup\{\theta_rr|r\in R\}$. Note that $H$ is isomorphic to $\langle X\rangle_{M(G)}\times \langle \Theta\rangle_{M(G)}$, since the only relations of $M(G)$ not containing $q$ give commutation between generators of $\langle X\rangle_{M(G)}$ and $\langle \Theta\rangle_{M(G)}$, and since $q$ does not appear in any element of $H$. 

We claim both $K_{-1}$ and $K_1$ are free. For the case of $K_{-1},$ consider the homomorphism $\varphi:H\to F(\Theta)$ defined by $\varphi(x)=1,$ $\varphi(\theta_\alpha)=\theta_\alpha$ for all $x\in~X$ and $\alpha\in X\cup R$. We now show the restriction of this map to $K_{-1}$ is an isomorphism. It is surjective because $\varphi(\theta_xx)=\theta_x$ and~$\varphi(\theta_r)=\theta_r.$ 

It is a much longer argument to  show that $\varphi|_{K_{-1}}$ is injective. Let $g \in \Ker(\varphi|_{K_{-1}})$. Applying the commutation relations $\theta_\alpha x=x\theta_\alpha$, . By the direct product structure of $H$, we can write this element as $g=_{M(G)}w\tau$ for some unique $w\in \langle X\rangle_{M(G)},\tau\in \langle \Theta\rangle_{M(G)}$. Since $K_{-1}$'s generators are of the form $\theta_xx$ or $\theta_r$, we can write $g$ as a product of these terms and their inverses. Before cancellation, every $x^{\pm1}$ has a ``corresponding" $\theta_x^{\pm1},$ and vice versa. Without loss of generality, suppose we write $g$ in this form with no letters cancelled. Note that, when we use commutation relations to write $g=_{M(G)}w\tau,$ the words $w$ and $\tau$ may not necessarily be reduced.

\paragraph{Claim:} Let $h\in K_{-1}.$ A letter $x^{\pm1}$ is to the left of another letter $y^{\pm1}$ before applying a commutation relation if and only if it is to the left afterwards, with the same holding for letters $\theta_x^{\pm1}$ and $\theta_y^{\pm1}.$ That is, applying commutation relations to $h$ preserves the relative order of the $X$-letters and the relative order of the $\theta$-letters. 
\begin{proof} If the commutation relation applied does not include $x^{\pm1}$ or $y^{\pm1},$ the claim obviously holds. Otherwise, suppose $x^{\pm1}$ is to the left of $y^{\pm1}$ and the relation involves $x^{\pm1}.$ Then, we can write either $h=w_1x^{\pm1}\theta_x^{\pm1}w_2y^{\pm1}w_3$ or $h=~w_1\theta_x^{\pm1}x^{\pm1}w_2y^{\pm1}w_3$ for some $w_1,w_2,w_3\in H$. Applying the commutation relation gives the words $w_1\theta_x^{\pm1}x^{\pm1}w_2y^{\pm1}w_3$ and  $w_1x^{\pm1}\theta_x^{\pm1}w_2y^{\pm1}w_3$ respectively, which both have $x^{\pm1}$ to the left of $y^{\pm1}.$ The converse, as well as the same claim for $\theta$-letters, both follow~similarly. \end{proof}

An immediate consequence is that, even after applying an arbitrary sequence of commutations to $g$ (without cancelling any letters) a letter $x^{\pm1}$ is to the left of a letter $y^{\pm1}$ if and only if the corresponding letter $\theta_x^{\pm1}$ is to the left of the corresponding~$\theta_y^{\pm1}.$

Returning to the proof that $K_{-1}$ is free, we have $\tau=\varphi(w\tau)=\varphi(g)=1.$  Since $\tau$ is a product of generators of the free group $\langle \Theta\rangle_{M(G)},$ this can only happen if all $\theta$-letters in $\tau$ cancel, or if $\tau$ was the trivial word to begin with. If the latter is the case, then $w$ is trivial as well, since any $X$-letter in $w$ would have a corresponding $\theta$-letter in $\tau.$ On the other hand, if $\tau$ has some $\theta$-letters, then all the $\theta$-letters in $\tau$ must cancel freely to get the trivial word. In particular, if $w$ has any $X$-letters, then their corresponding $\theta_x$'s must cancel. However, if we cancel a pair $\theta_x\theta_x^{-1},$ in $w$ there must be a pair $xx^{-1}$ which we can also cancel, hence we can freely cancel all letters in $w.$ In both cases, we get $w\equiv1,$ so $g=_{M(G)} 1$, hence $\varphi$ is injective.

The same argument (after observing that every $\theta_r$ has a corresponding $r$) shows that $K_1$ is also free, with the map $\varphi':H\to\langle \Theta\rangle_{M(G)}$ defined by $\theta_x~x~\mapsto~ \theta_x$ and $\theta_r~r~\mapsto~\theta_r$ giving the isomorphism, when restricted to $K_1$.

With this fact in hand, we define the map $$\kappa:\{\theta_xx,\theta_rr:x\in X,r\in R\}\to\{\theta_xx,\theta_r:x\in X,r\in R\}$$ by $\kappa(\theta_xx)= \theta_xx$ and $\kappa(\theta_rr)= \theta_r$. Since both groups are free, and $\kappa$ gives a bijection between their free generating sets, $\kappa$ can be extended to an isomorphism~$K_{1}\to~K_{-1}.$ Defining $M''(G)$ to be the HNN-extension of $H$ along $\kappa$, with~$q$~as the stable letter, we see that the $M''(G)$ is generated by $X\cup \Theta\cup \{q\}$ and subject to the relations $$\begin{matrix}\theta_\alpha x&=&x\theta_\alpha&& x\in X,&\alpha\in X\cup R\\ q\theta_xxq^{-1}&=&\theta_xx&&x\in X\ &\\q\theta_rrq^{-1}&=&\theta_r&&r\in R.\end{matrix}$$ The first relation follows from the direct product structure of $H$, and the other two from the definition of an HNN-extension. By construction, these determine all the relations of $M''(G).$ 

Since the first relation holds in both $M(G)$ and $M''(G)$, we see that the second relations of both groups are equivalent, since $$qx\theta_xq^{-1}=\theta_xx\iff q\theta_xx=\theta_xxq\iff \theta_xxq= q\theta_xx=qx\theta_x.$$ So are the third relations, since $$q\theta_rrq^{-1}=\theta_r\iff q\theta_rr=\theta_rq\iff qr\theta_r=\theta_rq\iff \theta_rq=qr\theta_r.$$ Thus, $M(G)$ is isomorphic $ M''(G)$.

To conclude this section, we record the following corollary of the above discussion.

\begin{corollary}\label{KintersectXq}
    Both $K_{-1}\cap\langle X,q\rangle_{M(G)}$ and $K_{1}\cap\langle X,q\rangle_{M(G)}$ are trivial.
\end{corollary}
\begin{proof}
    We prove only the first case, since the second holds by a similar argument. Define the homomorphism $\psi:M(G)\to\langle \Theta\rangle_{M(G)}$ by $\psi(x)=\varphi(q)=1$ for $x\in X,$ and $\psi(\theta_\alpha)=\theta_\alpha$ for $\theta_\alpha\in\Theta.$ This is an extension of the homomorphism $\varphi:H\to~\langle \Theta\rangle_{M(G)}$ constructed above, so it is injective on $K_{-1}.$ We have $\psi(\langle X,q\rangle)=\{1\},$ so $\psi(K_{-1}\cap\langle X,q\rangle_{M(G)})=\{1\}.$ By injectivity on $K_{-1},$ this implies $K_{-1}\cap~\langle X,q\rangle_{M(G)}=\{1\}$ as desired.
\end{proof}

\section{Additional preliminaries}\label{Preliminaries}
Recall that $\equiv$ denotes free equality between words.

\subsection{Dehn function}
Let $F(X)$ be the free group on $X$, and let $g$ be any word on $X$. Recall that $G=F(X)/\langle \langle R\rangle\rangle,$ where $\langle \langle R\rangle\rangle$ is the smallest normal subgroup containing $R$ in $F(X).$ It is a fact of group theory that $g=_G1$ if and only if $g\in \langle \langle R\rangle\rangle$ if and only if $g\equiv \prod_{i=1}^mw_ir_iw_i^{-1}$ for some $w_i\in F(x),r_i\in R^{\pm1}.$ We define $\delta(g)$ to be the smallest $m$ such that $g$ is expressible in this form, and the Dehn Function $\Delta$ to be $$\Delta(n)=\max\limits_{\begin{matrix}||g||\leq n\\g=_G1\end{matrix}}\delta(g).$$

  This notion has an equivalent definition as the maximum area of a minimal-area Van Kampen diagram with perimeter length $n.$ For an explanation of this geometric definition see \cite{bridson_salamon_bridson_2002}.

\subsection{$\Lambda$-length}\label{lamlength}For any word of the form $w=\prod_{j=1}^mw_jr_jw_j^{-1},$ with $r_j\in R^{\pm1}$ and $w_j\in F(x)$, let $$f(w)=m+||w_1||+||w_m||+\sum_{i=1}^{m-1}||w_i^{-1}w_{i+1}||.$$ For $g\in \langle \langle R\rangle\rangle,$ we define $\lambda(g)$ to be the smallest $f(w)$ such that $w\equiv g$. Essentially, this computes the length of $g$, allowing for free cancellation between adjacent $w_i$'s, but not between $w_i$'s and $r_i$'s. Below, we will see that this is precisely the number of steps needed to ``write" $g$ using the Miller Machine $M(G).$ We define $\Lambda(n)$ to be the largest $\lambda(g)$ such that $g=_G1 $ and $||g||\leq n$. 

This function has quadratic bounds in terms of $\Delta$, which we now prove. The second part of the proof relies on simple arguments using Van Kampen diagrams, but since we will not use these elsewhere, the geometric parts of the argument are only sketched.

\begin{proposition}\label{lambdadehnrel}
    Let $t=\max(\{||r||:r\in\langle\langle R\rangle\rangle \}\cup\{2\})$. For all $n$, we have the following bounds: $$\Delta(n)\leq\Lambda(n)\leq 3t(\Delta(n)+n)^2.$$
\end{proposition}
\begin{proof}
    
Firstly, observe that $\lambda(g)\geq \delta(g)$ for all $g\in\langle\langle R\rangle\rangle$ by definition, so $\Lambda(n)\geq \Delta(n)$. 

For the second inequality, let $g\in \langle\langle R\rangle\rangle$ and consider the word $w'$ representing $g$ with the smallest number of relators. The number of relators must be $\delta(g),$ and there exists a minimal-area Van Kampen diagram for $g$ corresponding to $w'$. Each $w_i$ represents a non self-intersecting path from the basepoint of this Van Kampen diagram to some point on the cell corresponding to $r_i.$ The number of cells of such a diagram is $\delta(g)$ and each cell has at most $t$ edges on its boundary. Moreover, every edge in the diagram is either on the boundary of a cell, or on the boundary of the diagram itself (or both). There are $||g||$ of the latter edges, since the word along the boundary freely equals $g$, anD at most $t\delta(g)$ of the former, so the diagram has at most $t\delta(g)+|g|$ edges, implying that $||w_i||$ is at most $t \delta(g)+||g||$. We compute $$f(w')= \delta(g)+||w_1||+||w_{\delta(g)}||+\sum_{i=1}^{\delta(g)-1}||w_{i}^{-1}w_{i+1}||$$ $$\leq \delta(g)+||w_1||+||w_{\delta(g)}||+\sum_{i=1}^{\delta(g)-1}||w_{i}|| + ||w_{i+1}||$$ $$\leq \delta(g)+t\delta(g)+||g||+t\delta(g)+||g||+\sum_{i=1}^{\delta(g)-1}2(t\delta(g)+||g||)$$ $$=\delta(g)+2\delta(g)(t\delta(g)+||g||).$$ Since $t$, $\delta(g)$, and $||g||$ are at least $1$, this in turn is bounded above by $3\delta(g)(t\delta(g)+||g||)\leq 3t(\delta(g)^2+\delta(g)||g||)\leq 3t(\delta(g)+||g||)^2 $ so passing to the maximum gives $\Lambda(n)\leq 3t(\Delta(n)+n)^2$ as desired.
\end{proof}

  \subsection{Conjugator length}\label{listconjdefs} Let $Z$ be either $G$ or $M(G)$. Given two elements $u,v\in Z,$ we say $u$ and $v$ are conjugate if there exists $\gamma\in Z$ such that $\gamma u\gamma^{-1}=v,$ and we denote this relation by $u\sim_Z v.$ For such $u\sim_Z v,$ we define $c_Z(u,v)$ as the minimal length of a conjugator taking $u$ to $v$: $$c_Z(u,v)=\min\{|\gamma|\mid \gamma u\gamma^{-1}=_Zv\}.$$ For the group $Z$ as a whole, we define the conjugator length function $\Gamma_Z(n)$ by $$\Gamma_Z(n)=\max\{c_Z(u,v)\mid ||u||+||v||\leq n,\ u\sim_Z v\}.$$ In the group $M(G),$ we will be particularly interested in when $u$ and $v$ are conjugate via elements of $\langle X,\Theta\rangle_{M(G)}$. That is, when $\gamma u\gamma^{-1}=v$ for some $\gamma$ not containing $q.$ We denote this relation by $u\approx v$, and define the functions $$c'_{M(G)}(u,v)=\min\{|\gamma|\mid \gamma u\gamma^{-1}=_{M(G)}v,\ \gamma\in \langle X\cup \Theta\rangle_{M(G)}\},$$ 
$$\Gamma'_{M(G)}(n)=\max\{c'_{M(G)}(u,v)\mid ||u||+||v||\leq n,\ u\approx v\}.$$

Next, let $k\geq 1$ and $\sigma=(\sigma_1,...,\sigma_k)$, where $\sigma_i\in \{\pm1\}$.  We define $D_{k,\sigma},D_{k,\sigma}'$ similarly to $\Gamma_{M(G)},\Gamma'_{M(G)}$ respectively, except that the maxima are restricted to conjugate pairs of the form $x=q^{\sigma_1}u_1q^{\sigma_2}...q^{\sigma_k}u_k$, $y=q^{\sigma_1}v_1q^{\sigma_2}...q^{\sigma_k}v_k$ for some $u_1,...,u_k,v_1,...,v_k\in \langle X\rangle_{M(G)}$ such that the above words are reduced and $\sum_i||u_i||+||v_i||\leq n-2k$. Note that the total number of $q$-letter in $x$ and $y$ is $2k$, so if the above words are reduced, then $||x||+||y||\leq n$ if and only if $\sum_i||u_i||+||v_i||\leq n-2k$. Also, we define $D_0,D_0'$ in the same way, with the maxima restricted to conjugate pairs of the form $(qu,q),$ where $u\in\langle X\rangle_{M(G)}$ is such that $||u||\leq n-2$.


\subsection{Iso-computational list conjugacy}\label{ICLCdef} Let $\textbf{u}=~(u_1,...,u_k)$ and $\textbf{v}=~(v_1,...,v_k)$ be $k$-tuples of words on $X$, let $\sigma=~(\sigma_1,...,\sigma_k)$ be a tuple of elements of $\{\pm1\}$, and suppose $\sigma_i=-\sigma_{i+1}$ implies $u_i\not\equiv ~1\not\equiv v_i$ for $i<k$. We write $\textbf{u}\overset{\sigma}{\sim}\textbf{v}$ if there exists words $w,\varepsilon$ on the generators of $G$ such that $\varepsilon$ represents the identity of $G$ and the following are true: \begin{itemize}
    \item If $\sigma_i=\sigma_{i+1}=1$, then $w\varepsilon u_i w^{-1}\equiv v_i$
    \item If $\sigma_i=1,\sigma_{i+1}=-1$, then $w\varepsilon u_i \varepsilon^{-1}w^{-1}\equiv v_i$
    \item If $\sigma_{i}=-1,\sigma_{i+1}=1$, then $wu_iw^{-1}\equiv v_i$
    \item If $\sigma_i=-1,\sigma_{i+1}=-1$, then $wu_i\varepsilon^{-1}w^{-1}\equiv v_i$
\end{itemize} If this is so, we say that $\textbf{u}$ and $\textbf{v}$ are \textit{iso-computationally list-conjugate via $\sigma$}. The intuition behind this definition is that $\varepsilon$ represents computation showing that words are equal in $G$, at least for the components of the tuple where $\sigma_i=\sigma_{i+1}$ - the name is not entirely accurate for the other cases, however we still find it suggestive. 

Now, fix some $\sigma$ and $k$ as above. For any $k$-tuples $\textbf{u},\textbf{v}$ such that $\textbf{u}\overset{\sigma}{\sim}\textbf{v}$, we define $c_{k,\sigma}(\textbf{u},\textbf{v})$ as the minimum value of $||w||$ for all $w,\varepsilon$ satisfying the above condition. Next, we define $$\textbf{C}_{k,\sigma}(n)=\max\{c_{k,\sigma}(\textbf{u},\textbf{v}) \mid \textbf{u}\overset{\sigma}{\sim}\textbf{v},\ \sum_{j}||u_j||+||v_j||\leq n-2k\}.$$ Note that this functions depends on the presentation of $G$ chosen, so (like $M(G)$) $\textbf{C}_{k,\sigma}$ is not well-defined for a group as a whole.

\subsection{Two miscellaneous notions} Given any words $a$ and $b$ on $X$, we know that $b\in \langle a\rangle_{M(G)}$ if and only if $b=_{M(G)}a^k$ for some $k\in \ZZ$. The problem of deciding whether there exists such a $k$, we call the \textit{Cyclic Subgroup Membership problem}.  
Lastly, adopting the convention that $\sigma_{k+1}=\sigma_1$, if a tuple $\sigma=(\sigma_1,...,\sigma_k)$ is such that $\sigma_i\in \{\pm1\}$ and $\sigma_i=-\sigma_{i+1}$ for all $i=~1,...,k$, then we say $\sigma$ is \textit{alternating}.
\section{Diagrams}\label{Diag} One of the most powerful tools we will use in this paper is annular diagrams. A pair of words $u$ and $v$ over a set $S$ represent conjugate elements in a finitely presented group generated by $S$ if and only if there exists a planar annular diagram whose inner boundary is labelled $v$, whose outer boundary is labeled $u$, and whose interior is filled by cells corresponding to the set of relations in the group presentation, as seen in Figure \ref{fig:generic_annular_diagram}. We call $u$ and $v$, respectively, the ``outer" and ``inner" words along the boundary of the diagram, and say that this diagram ``witnesses the conjugacy of $u$ and $v$."
\begin{figure}
    \centering
\includegraphics[scale=.5]{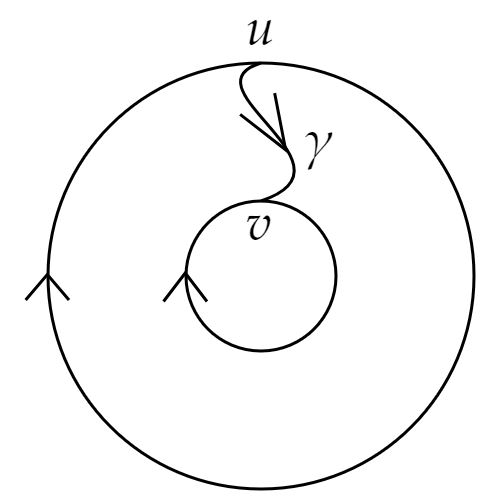}
    \caption{A typical annular diagram witnessing $u\sim v$}
    \label{fig:generic_annular_diagram}
\end{figure}

For general information on annular diagrams, we refer the reader to \cite{lyndon2001combinatorial}. In this section, we will describe some important observations about these diagrams in the case of $M(G)$. 

In particular, for any $x$ and $y\in M(G)$, $x\mgcon y$ if and only if there exists a planar annular diagram whose boundary components are labeled by words representing $x$ and $y$, and whose interior is filled in by cells corresponding to the defining relations, shown in Figure \ref{fig:MG_cells} (recall that there is a different relation for each $x,y\in X,r\in R$, and $\alpha\in X\cup R$).

  \begin{figure}
         \centering

\tikzset{every picture/.style={line width=0.75pt}} 

\begin{tikzpicture}[x=0.75pt,y=0.75pt,yscale=-1,xscale=1]

\draw   (7.22,85.02) -- (79.43,85.02) -- (79.43,172.58) -- (7.22,172.58) -- cycle ;
\draw   (36.11,76.26) -- (50.55,85.02) -- (36.11,93.77) ;
\draw   (86.66,120.04) -- (79.43,137.55) -- (72.21,120.04) ;
\draw   (36.11,163.82) -- (50.55,172.58) -- (36.11,181.33) ;
\draw   (14.44,120.04) -- (7.22,137.55) -- (0,120.04) ;
\draw   (127.82,85.02) -- (272.24,85.02) -- (272.24,172.58) -- (127.82,172.58) -- cycle ;
\draw [line width=1.5]    (200.03,76.26) -- (200.03,93.77) ;
\draw [line width=1.5]    (200.03,163.82) -- (200.03,181.33) ;
\draw   (156.7,76.26) -- (171.15,85.02) -- (156.7,93.77) ;
\draw   (156.7,163.82) -- (171.15,172.58) -- (156.7,181.33) ;
\draw   (228.92,163.82) -- (243.36,172.58) -- (228.92,181.33) ;
\draw   (228.92,76.26) -- (243.36,85.02) -- (228.92,93.77) ;
\draw   (135.04,120.04) -- (127.82,137.55) -- (120.6,120.04) ;
\draw   (279.47,120.04) -- (272.24,137.55) -- (265.02,120.04) ;
\draw   (308.35,85.02) -- (452.78,85.02) -- (452.78,172.58) -- (308.35,172.58) -- cycle ;
\draw [line width=1.5]    (380.57,163.82) -- (380.57,181.33) ;
\draw   (373.34,76.26) -- (387.79,85.02) -- (373.34,93.77) ;
\draw   (337.24,163.82) -- (351.68,172.58) -- (337.24,181.33) ;
\draw   (409.45,163.82) -- (423.89,172.58) -- (409.45,181.33) ;
\draw   (315.57,120.04) -- (308.35,137.55) -- (301.13,120.04) ;
\draw   (460,120.04) -- (452.78,137.55) -- (445.56,120.04) ;

\draw (37.24,54.44) node [anchor=north west][inner sep=0.75pt]  [font=\small]  {$y$};
\draw (37.24,182.95) node [anchor=north west][inner sep=0.75pt]  [font=\small]  {$y$};
\draw (20,122.4) node [anchor=north west][inner sep=0.75pt]  [font=\small]  {$\theta _{\alpha }$};
\draw (51,122.4) node [anchor=north west][inner sep=0.75pt]  [font=\small]  {$\theta _{\alpha }$};
\draw (137,122.4) node [anchor=north west][inner sep=0.75pt]  [font=\small]  {$q$};
\draw (256.86,121.19) node [anchor=north west][inner sep=0.75pt]  [font=\small]  {$q$};
\draw (317,122.4) node [anchor=north west][inner sep=0.75pt]  [font=\small]  {$q$};
\draw (437.39,121.19) node [anchor=north west][inner sep=0.75pt]  [font=\small]  {$q$};
\draw (154.73,58.02) node [anchor=north west][inner sep=0.75pt]  [font=\small]  {$\theta _{x}$};
\draw (226.94,182.36) node [anchor=north west][inner sep=0.75pt]  [font=\small]  {$\theta _{x}$};
\draw (226.53,59.9) node [anchor=north west][inner sep=0.75pt]  [font=\small]  {$x$};
\draw (160.09,182.48) node [anchor=north west][inner sep=0.75pt]  [font=\small]  {$x$};
\draw (370.78,58.46) node [anchor=north west][inner sep=0.75pt]  [font=\small]  {$ \begin{array}{l}
\theta _{r}\\
\end{array}$};
\draw (336.29,182.48) node [anchor=north west][inner sep=0.75pt]  [font=\small]  {$r$};
\draw (407.61,182.36) node [anchor=north west][inner sep=0.75pt]  [font=\small]  {$\theta _{r}$};

\end{tikzpicture}
\caption{The three types cells in an annular diagram for $M(G)$}
\label{fig:MG_cells}

\end{figure}
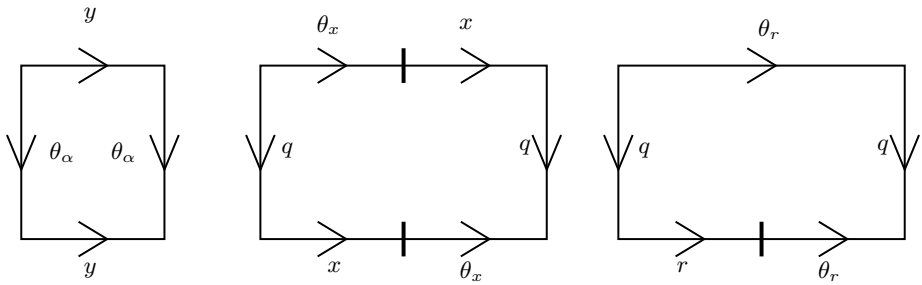

 We have drawn all the relations as rectangles with $q$'s or $\theta_\alpha$'s on either side, facing in the same direction. If a $q$-edge is present anywhere in the diagram, it must be part of a corridor (which we will call a $q$-corridor) made up of the second and third types of cell drawn in Figure \ref{fig:MG_cells}. These corridors must either go from the boundary of the annulus to itself, or form some closed loop in the inside. All the possibilities are shown in Figure \ref{fig: four_types_of_corridors}.

\begin{figure}

\centering
\tikzset{every picture/.style={line width=0.75pt}} 

\begin{tikzpicture}[x=0.3pt,y=0.3pt,yscale=-1,xscale=1]

\draw   (80,183.23) .. controls (80,85.6) and (160.59,6.46) .. (260,6.46) .. controls (359.41,6.46) and (440,85.6) .. (440,183.23) .. controls (440,280.86) and (359.41,360) .. (260,360) .. controls (160.59,360) and (80,280.86) .. (80,183.23) -- cycle ;
\draw   (214.57,183.23) .. controls (214.57,156.74) and (234.91,135.27) .. (260,135.27) .. controls (285.09,135.27) and (305.43,156.74) .. (305.43,183.23) .. controls (305.43,209.72) and (285.09,231.2) .. (260,231.2) .. controls (234.91,231.2) and (214.57,209.72) .. (214.57,183.23) -- cycle ;
\draw  [draw opacity=0] (237.95,141.3) .. controls (237.79,140.21) and (237.71,139.09) .. (237.71,137.96) .. controls (237.71,124.86) and (248.46,114.25) .. (261.71,114.25) .. controls (274.97,114.25) and (285.71,124.86) .. (285.71,137.96) .. controls (285.71,139.76) and (285.51,141.52) .. (285.13,143.21) -- (261.71,137.96) -- cycle ; \draw   (237.95,141.3) .. controls (237.79,140.21) and (237.71,139.09) .. (237.71,137.96) .. controls (237.71,124.86) and (248.46,114.25) .. (261.71,114.25) .. controls (274.97,114.25) and (285.71,124.86) .. (285.71,137.96) .. controls (285.71,139.76) and (285.51,141.52) .. (285.13,143.21) ;  
\draw  [draw opacity=0] (243.73,138.73) .. controls (243.72,138.47) and (243.71,138.22) .. (243.71,137.96) .. controls (243.71,128.14) and (251.77,120.18) .. (261.71,120.18) .. controls (271.66,120.18) and (279.71,128.14) .. (279.71,137.96) .. controls (279.71,138.66) and (279.67,139.36) .. (279.59,140.04) -- (261.71,137.96) -- cycle ; \draw   (243.73,138.73) .. controls (243.72,138.47) and (243.71,138.22) .. (243.71,137.96) .. controls (243.71,128.14) and (251.77,120.18) .. (261.71,120.18) .. controls (271.66,120.18) and (279.71,128.14) .. (279.71,137.96) .. controls (279.71,138.66) and (279.67,139.36) .. (279.59,140.04) ;  
\draw  [draw opacity=0] (285.41,7.74) .. controls (285.41,8.03) and (285.42,8.33) .. (285.41,8.62) .. controls (285.24,21.71) and (274.36,32.2) .. (261.11,32.03) .. controls (247.85,31.86) and (237.25,21.11) .. (237.41,8.02) .. controls (237.42,7.95) and (237.42,7.88) .. (237.42,7.81) -- (261.41,8.32) -- cycle ; \draw   (285.41,7.74) .. controls (285.41,8.03) and (285.42,8.33) .. (285.41,8.62) .. controls (285.24,21.71) and (274.36,32.2) .. (261.11,32.03) .. controls (247.85,31.86) and (237.25,21.11) .. (237.41,8.02) .. controls (237.42,7.95) and (237.42,7.88) .. (237.42,7.81) ;  
\draw  [draw opacity=0] (279.38,7.25) .. controls (279.41,7.67) and (279.42,8.11) .. (279.41,8.54) .. controls (279.29,18.36) and (271.12,26.23) .. (261.18,26.1) .. controls (251.24,25.98) and (243.29,17.91) .. (243.41,8.09) .. controls (243.42,7.87) and (243.42,7.64) .. (243.44,7.42) -- (261.41,8.32) -- cycle ; \draw   (279.38,7.25) .. controls (279.41,7.67) and (279.42,8.11) .. (279.41,8.54) .. controls (279.29,18.36) and (271.12,26.23) .. (261.18,26.1) .. controls (251.24,25.98) and (243.29,17.91) .. (243.41,8.09) .. controls (243.42,7.87) and (243.42,7.64) .. (243.44,7.42) ;  
\draw    (305.43,186.68) -- (440,187.54) ;
\draw    (304,170.3) -- (440,170.3) ;
\draw   (105.71,170.3) .. controls (105.71,160.77) and (113.39,153.05) .. (122.86,153.05) .. controls (132.32,153.05) and (140,160.77) .. (140,170.3) .. controls (140,179.82) and (132.32,187.54) .. (122.86,187.54) .. controls (113.39,187.54) and (105.71,179.82) .. (105.71,170.3) -- cycle ;
\draw   (97.14,170.3) .. controls (97.14,156.01) and (108.66,144.43) .. (122.86,144.43) .. controls (137.06,144.43) and (148.57,156.01) .. (148.57,170.3) .. controls (148.57,184.58) and (137.06,196.17) .. (122.86,196.17) .. controls (108.66,196.17) and (97.14,184.58) .. (97.14,170.3) -- cycle ;
\draw  [draw opacity=0] (357.5,197.91) .. controls (350.45,245.71) and (309.48,282.39) .. (260,282.39) .. controls (205.56,282.39) and (161.43,238) .. (161.43,183.23) .. controls (161.43,128.46) and (205.56,84.07) .. (260,84.07) .. controls (307.23,84.07) and (346.71,117.49) .. (356.33,162.1) -- (260,183.23) -- cycle ; \draw   (357.5,197.91) .. controls (350.45,245.71) and (309.48,282.39) .. (260,282.39) .. controls (205.56,282.39) and (161.43,238) .. (161.43,183.23) .. controls (161.43,128.46) and (205.56,84.07) .. (260,84.07) .. controls (307.23,84.07) and (346.71,117.49) .. (356.33,162.1) ;  
\draw  [draw opacity=0] (350.01,196.78) .. controls (343.29,239.16) and (305.55,271.62) .. (260,271.62) .. controls (209.7,271.62) and (168.93,232.04) .. (168.93,183.23) .. controls (168.93,134.42) and (209.7,94.85) .. (260,94.85) .. controls (303.4,94.85) and (339.71,124.31) .. (348.85,163.74) -- (260,183.23) -- cycle ; \draw   (350.01,196.78) .. controls (343.29,239.16) and (305.55,271.62) .. (260,271.62) .. controls (209.7,271.62) and (168.93,232.04) .. (168.93,183.23) .. controls (168.93,134.42) and (209.7,94.85) .. (260,94.85) .. controls (303.4,94.85) and (339.71,124.31) .. (348.85,163.74) ;  
\draw  [dash pattern={on 0.84pt off 2.51pt}]  (348.85,163.74) -- (350.01,196.78) ;
\draw  [dash pattern={on 0.84pt off 2.51pt}]  (356.33,162.1) -- (357.5,197.91) ;

\draw (286,115) node [anchor=north west][inner sep=0.75pt]  [font=\small] [align=left] {A};
\draw (252.86,35) node [anchor=north west][inner sep=0.75pt]  [font=\small] [align=left] {A};
\draw (150,100) node [anchor=north west][inner sep=0.75pt]  [font=\small] [align=left] {D};
\draw (390,140) node [anchor=north west][inner sep=0.75pt]  [font=\small] [align=left] {C};
\draw (97.14,200) node [anchor=north west][inner sep=0.75pt]  [font=\small] [align=left] {B};

\end{tikzpicture}
\caption{The four types ofcorridors in an annular diagram for $M(G)$}
\label{fig: four_types_of_corridors}
\end{figure}
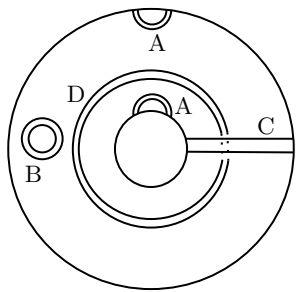

As reflected in this figure, we say a $q$-corridor is of type A, B, C, or D, depending on which of the above possibilities it falls into. In our arguments below, we also write ``$q$-corridors" when the type is understood from the context, or irrelevant to the purpose at hand. The planarity of the diagram implies no two $q$-corridors intersect, so the $q$-corridor D cannot coexist with C (this is why we have drawn part of corridor D with dashes). 

We now make two observations on $q$-corridors, both apparent by inspection of the relations.

\begin{remark} All the $q$-edges inside a $q$-corridor must be directed towards the same boundary component of the corridor. The word along that component must be in $K_1,$ and the word along the other must be in $K_{-1}.$ We will call these the ``boundary words" of the corridor. As a matter of convention, we consider the boundary words to be the words along the boundary component \textit{before} freely reducing.
\end{remark}
\begin{remark}  Every cell has a $\theta$-letter on both sides. If two of them cancel, then the cells words along their corresponding cells must cancel in their entirety, and so we can remove them according to the diagrams in Figure \ref{fig:qcor_cancellation}, and still have a diagram with the same words along the boundary components. Note that no paths are lengthened by this procedure.

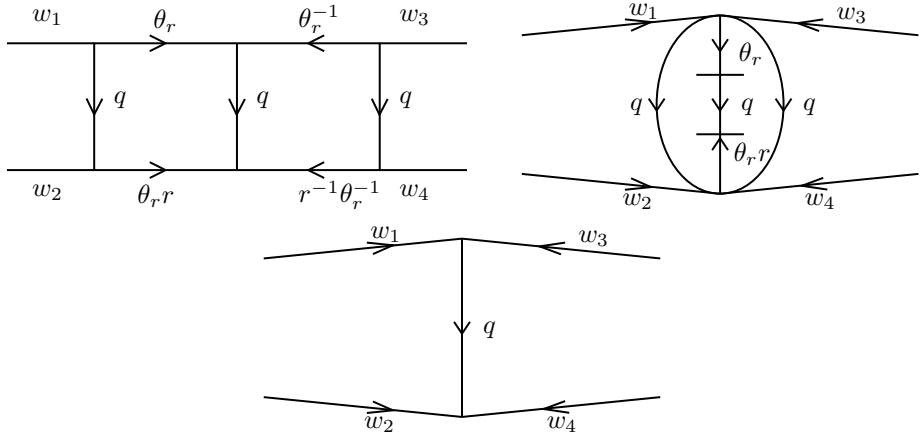
\begin{figure}
    \centering
$\begin{matrix}

\tikzset{every picture/.style={line width=0.75pt}} 

\begin{tikzpicture}[x=0.3pt,y=0.3pt,yscale=-1,xscale=1]

\draw   (130,380) -- (130,220) -- (490,220) -- (490,380) -- (130,380) -- cycle ;
\draw    (310,220) -- (310,380) ;
\draw    (20,220) -- (130,220) ;
\draw    (20,380) -- (130,380) ;
\draw    (490,220) -- (600,220) ;
\draw    (490,380) -- (600,380) ;
\draw   (200,210) -- (220,220) -- (200,230) ;
\draw   (200,370) -- (220,380) -- (200,390) ;
\draw   (420,390) -- (399.96,380.07) -- (419.93,370) ;
\draw   (420,230) -- (399.96,220.07) -- (419.93,210) ;
\draw   (140.31,290) -- (130,309.85) -- (120.31,289.7) ;
\draw   (320,290.31) -- (309.69,310.15) -- (300,290) ;
\draw   (500,290.31) -- (489.69,310.15) -- (480,290) ;

\draw (201,173.4) node [anchor=north west][inner sep=0.75pt] {$\theta _{r}$};
\draw (384,168) node [anchor=north west][inner sep=0.75pt] {$\theta _{r}^{-1}$};
\draw (181,395) node [anchor=north west][inner sep=0.75pt] {$\theta _{r} r$};
\draw (384,390) node [anchor=north west][inner sep=0.75pt] {$r^{-1} \theta _{r}^{-1}$};
\draw (48,173.4) node [anchor=north west][inner sep=0.75pt] {$w_{1}$};
\draw (48,395) node [anchor=north west][inner sep=0.75pt] {$w_{2}$};
\draw (511,173.4) node [anchor=north west][inner sep=0.75pt] {$w_{3}$};
\draw (511,395) node [anchor=north west][inner sep=0.75pt] {$w_{4}$};
\draw (151,277.4) node [anchor=north west][inner sep=0.75pt] {$q$};
\draw (331,277.4) node [anchor=north west][inner sep=0.75pt] {$q$};
\draw (511,277.4) node [anchor=north west][inner sep=0.75pt] {$q$};

\end{tikzpicture}

&&

\tikzset{every picture/.style={line width=0.75pt}} 

\begin{tikzpicture}[x=0.3pt,y=0.3pt,yscale=-.5,xscale=1]

\draw    (300,50) -- (300,500) ;
\draw    (300,50) -- (50,100) ;
\draw    (50,450) -- (300,500) ;
\draw    (300,50) -- (550,100) ;
\draw    (550,450) -- (300,500) ;
\draw  [draw opacity=0] (300,500) .. controls (300,500) and (300,500) .. (300,500) .. controls (255.82,500) and (220,399.26) .. (220,275) .. controls (220,150.74) and (255.82,50) .. (300,50) -- (300,275) -- cycle ; \draw   (300,500) .. controls (300,500) and (300,500) .. (300,500) .. controls (255.82,500) and (220,399.26) .. (220,275) .. controls (220,150.74) and (255.82,50) .. (300,50) ;  
\draw  [draw opacity=0] (300,50) .. controls (300,50) and (300,50) .. (300,50) .. controls (344.18,50) and (380,150.74) .. (380,275) .. controls (380,399.26) and (344.18,500) .. (300,500) -- (300,275) -- cycle ; \draw   (300,50) .. controls (300,50) and (300,50) .. (300,50) .. controls (344.18,50) and (380,150.74) .. (380,275) .. controls (380,399.26) and (344.18,500) .. (300,500) ;  
\draw   (180,60) -- (212.86,66.72) -- (187.76,88.98) ;
\draw   (187.76,461.02) -- (212.86,483.28) -- (180,490) ;
\draw   (424.79,89.54) -- (397.85,69.56) -- (430,60) ;
\draw   (435.21,489.54) -- (403.06,479.98) -- (430,460) ;
\draw   (230,260) -- (220,290) -- (210,260) ;
\draw   (390,260) -- (380,290) -- (370,260) ;
\draw    (270,200) -- (330,200) ;
\draw    (270,350) -- (330,350) ;
\draw   (310,110) -- (300,140) -- (290,110) ;
\draw   (310,260) -- (300,290) -- (290,260) ;
\draw   (290,390) -- (300,360) -- (310,390) ;

\draw (181,13.4) node [anchor=north west][inner sep=0.75pt] {$w_{1}$};
\draw (173,492.4) node [anchor=north west][inner sep=0.75pt] {$w_{2}$};
\draw (403,492.4) node [anchor=north west][inner sep=0.75pt] {$w_{4}$};
\draw (443,22.4) node [anchor=north west][inner sep=0.75pt] {$w_{3}$};
\draw (183,252.4) node [anchor=north west][inner sep=0.75pt] {$q$};
\draw (401,252.4) node [anchor=north west][inner sep=0.75pt] {$q$};
\draw (320,112.4) node [anchor=north west][inner sep=0.75pt] {$\theta _{r}$};
\draw (297,353.4) node [anchor=north west][inner sep=0.75pt] {$ \begin{array}{l}
\theta _{r} r\\
\end{array}$};
\draw (323,252.4) node [anchor=north west][inner sep=0.75pt] {$q$};

\end{tikzpicture}

\end{matrix}$

\tikzset{every picture/.style={line width=0.75pt}} 

\begin{tikzpicture}[x=0.3pt,y=0.3pt,yscale=-.5,xscale=1]

\draw    (300,50) -- (300,500) ;
\draw    (300,50) -- (50,100) ;
\draw    (50,450) -- (300,500) ;
\draw    (300,50) -- (550,100) ;
\draw    (550,450) -- (300,500) ;
\draw   (180,60) -- (212.86,66.72) -- (187.76,88.98) ;
\draw   (187.76,461.02) -- (212.86,483.28) -- (180,490) ;
\draw   (424.79,89.54) -- (397.85,69.56) -- (430,60) ;
\draw   (435.21,489.54) -- (403.06,479.98) -- (430,460) ;
\draw   (310,260) -- (300,290) -- (290,260) ;

\draw (181,13.4) node [anchor=north west][inner sep=0.75pt] {$w_{1}$};
\draw (173,492.4) node [anchor=north west][inner sep=0.75pt] {$w_{2}$};
\draw (403,492.4) node [anchor=north west][inner sep=0.75pt] {$w_{4}$};
\draw (443,22.4) node [anchor=north west][inner sep=0.75pt] {$w_{3}$};
\draw (323,252.4) node [anchor=north west][inner sep=0.75pt] {$q$};

\end{tikzpicture}
    \caption{Cancelling along a $q$-corridor}
    \label{fig:qcor_cancellation}
\end{figure}

Of course, the same can be done for the case of $\theta_xxx^{-1}\theta_x^{-1}.$ This removal only shortens paths in the diagram, so doing this will only shorten the length of the conjugator $\gamma$ or keep it the same. This means that we can assume the words along $q$-corridors are reduced in the generators $K_{-1}$ or $K_1$. (Of course, they are not necessarily reduced in $H$, since if $X=\{a,b\}$ and $R=\{ab\}$, then the word $\theta_bbb^{-1}a^{-1}\theta_{ab}^{-1}\in K_1$ is not reduced.)
\end{remark}

We also have $\theta$-corridors in our diagrams, which similarly do not intersect each other (although a $\theta$-corridor can intersect a $q$-corridor). These are of the same types as $q$-corridors, and when necessary we will refer to them accordingly. This leads us to one final geometric lemma:

\begin{lemma}\label{thcorwrap} Suppose we are given an annular diagram whose inner and outer boundary words do not contain any $\theta$-letter, and which has at least one $q$-corridor $Q$ of type C whose boundary words are reduced in $K_{-1}$ and $K_1$. Then, for every $\theta$-edge $e$ on the boundary of $Q$ there exists a $\theta$-corridor of type D which contains $e$ and intersects $Q$ exactly once.
\end{lemma}
\begin{proof} Note that every $\theta$-edge in this diagram must be part of some $\theta$-corridor, so we need only show that the corridor containing $e$ has the desired properties. By our assumptions, the only $\theta$-corridors in this diagram are of type D and B. We first claim that no $\theta$-corridor of type $B$ can cross $Q$. Indeed, if one does so, it must do so at two $\theta$-edges $e$ and $e'$, which have the same label and opposite orientations. If a $\theta$-edge lies between $e$ and $e',$ every letter on that side of its cell must lie between $e$ and $e'$ as well. We can therefore write the word between $e$ and $e'$ as $w_1\delta w_2,$ where $\delta\in K_{\pm 1}$ and $w_1,w_2\in\langle X\rangle_{M(G)}.$ If $w'\in \langle X,q\rangle_{M(G)}$ is the boundary word along the $\theta$-corridor, we must have $w_1\delta w_2=_{M(G)} w',$ or rather, $\delta=_{M(G)} w_1^{-1}w'w_2^{-1}.$ But $w_1^{-1}w'w_2^{-1}\in\langle X,q\rangle_{M(G)},$ so Corollary \ref{KintersectXq} implies $\delta=_{M(G)}1$. Thus, either $\delta$ is not reduced in $\pm1$, or it is the trivial word. The latter implies the $\theta$-corridor crosses at adjacent edges, so either case contradicts the assumption that the words along $Q$ are reduced. This situation is depicted in the Figure \ref{fig:corridor_crossing}.

\begin{figure}

    \centering

\tikzset{every picture/.style={line width=0.75pt}} 

\begin{tikzpicture}[x=0.5pt,y=0.5pt,yscale=-1,xscale=1]

\draw   (10,245) -- (81.11,245) -- (81.11,290.56) -- (10,290.56) -- cycle ;
\draw   (81.11,245) -- (152.22,245) -- (152.22,290.56) -- (81.11,290.56) -- cycle ;
\draw   (152.22,245) -- (223.33,245) -- (223.33,290.56) -- (152.22,290.56) -- cycle ;
\draw   (578.89,245) -- (650,245) -- (650,290.56) -- (578.89,290.56) -- cycle ;
\draw  [draw opacity=0] (10,245) .. controls (10,245) and (10,245) .. (10,245) .. controls (10,131.78) and (153.27,40) .. (330,40) .. controls (506.73,40) and (650,131.78) .. (650,245) -- (330,245) -- cycle ; \draw   (10,245) .. controls (10,245) and (10,245) .. (10,245) .. controls (10,131.78) and (153.27,40) .. (330,40) .. controls (506.73,40) and (650,131.78) .. (650,245) ;  
\draw  [draw opacity=0][line width=3]  (50.63,245) .. controls (50.63,157.14) and (175.71,85.91) .. (330,85.91) .. controls (484.29,85.91) and (609.37,157.14) .. (609.37,245) -- (330,245) -- cycle ; \draw  [ draw opacity=.75 ][line width=3]  (50.63,245) .. controls (50.63,157.14) and (175.71,85.91) .. (330,85.91) .. controls (484.29,85.91) and (609.37,157.14) .. (609.37,245) ;  
\draw [ draw opacity=.5 ][line width=3]    (50.63,245) -- (81.11,245) ;
\draw [ draw opacity=.5 ][line width=3]    (578.89,245) -- (609.37,245) ;
\draw [draw opacity=1 ][line width=3]    (81.11,245) -- (578.89,245) ;
\draw    (27.33,178.67) -- (60,200) ;
\draw    (62,133) -- (89.67,161.67) ;
\draw    (223.33,290.56) -- (580,290) ;

\draw (309.86,216.48) node [anchor=north west][inner sep=0.75pt]  [font=\large]  {$w'$};
\draw (315.15,92.91) node [anchor=north west][inner sep=0.75pt]  [font=\large]  {$\delta $};
\draw (45,252.4) node [anchor=north west][inner sep=0.75pt]  [font=\large]  {$w_{1}$};
\draw (581,252.4) node [anchor=north west][inner sep=0.75pt]  [font=\large]  {$w_{2}$};
\draw (47,168.4) node [anchor=north west][inner sep=0.75pt]  [font=\large]  {$q$};
\draw (77,122.4) node [anchor=north west][inner sep=0.75pt]  [font=\large]  {$q$};
\draw (92.31,124.05) node [anchor=north west][inner sep=0.75pt]  [rotate=-320]  {$...$};
\draw (121,58) node [anchor=north west][inner sep=0.75pt]    {$Q$};
\draw (84,262.4) node [anchor=north west][inner sep=0.75pt]  [font=\large]  {$\theta _{\alpha }$};
\draw (154,262.4) node [anchor=north west][inner sep=0.75pt]  [font=\large]  {$\theta _{\alpha }$};
\draw (224,262.4) node [anchor=north west][inner sep=0.75pt]  [font=\large]  {$\theta _{\alpha } \ \ \ ...$};
\draw (24.33,218.4) node [anchor=north west][inner sep=0.75pt]  [font=\large]  {$q$};

\end{tikzpicture}
    \caption{Corridor crossing}

    \label{fig:corridor_crossing}
\end{figure}
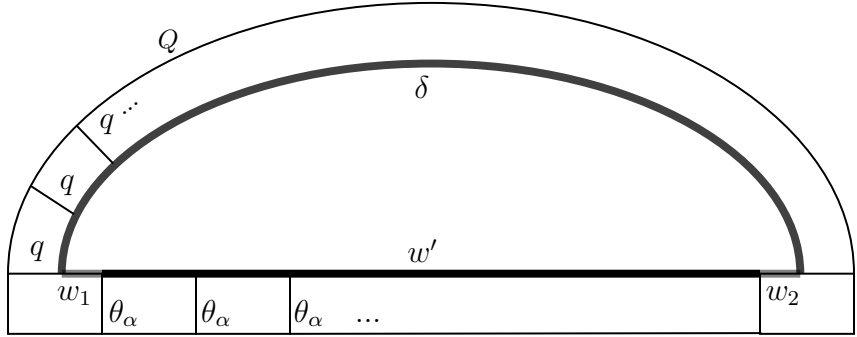

Now, if $\delta$ has no $\theta$-letters, then we can cancel the letters corresponding to $e$ and $e'$ using commutation relations. By the discussion in Subsection \ref{stabletterqHNN}, this contradicts our assumption that the words along $Q$'s boundary are reduced in $K_{\pm 1}.$ Alternatively, if it does have $\theta$-letters, then we still have $\varphi'(\delta)=1,$ where $\varphi'$ is the extended homomorphism constructed in the proof of Corollary \ref{KintersectXq}. This implies that the $\theta$-letters can freely cancel in $\langle X\rangle_{M(G)}.$ Since $\delta$ only contains $\theta$-letters and $X$-letters, which commute, the $\theta$-letters in $\delta$ can therefore be cancelled using commutation relations in $M(G),$ which gives the same contradiction. Therefore, there is no $\theta$-corridor of type B.

 Finally, if a $\theta$-corridor of type D crosses $Q$ more than once, then at some point it must ``backtrack" across $Q$. This will create a region bounded on two sides, one of them a $q$-corridor and the other a $\theta$-corridor, with the two crossing at $\theta$-edges. This gives the same contradiction reached above.

\end{proof}

Our discussion of $q$ and $\theta$-corridors is a special case of Collins' Lemma, and we refer the reader to Part III of \cite{brady2007geometry} for more general details.

\section{Conjugation in $M(G)$}\label{conjMG}
In this section, we analyse the relation $\approx$, which we will use to prove Theorem A in the next section. Recall that, for any two $k$-tuples $\textbf{u}=(u_1,...,u_k)$ and $ \textbf{v}=(v_1,...,v_k)$ of words on $X$, and any $k$-tuple $\sigma=(\sigma_1,...,\sigma_k)$ of elements of $\{\pm1\}$, we write $\textbf{u}\overset{\sigma}{\sim}\textbf{v}$ if $\sigma_i=-\sigma_{i+1}$ implies $u_i\not\equiv 1\not\equiv v_i$ for $i<k$, and there exists words $w,\varepsilon$ on the generators of $G$ such that $\varepsilon=_G1$ and: \begin{itemize}
    \item If $\sigma_i=\sigma_{i+1}=1$, then $w\varepsilon u_i w^{-1}\equiv v_i$
    \item If $\sigma_i=1,\sigma_{i+1}=-1$, then $w\varepsilon u_i \varepsilon^{-1}w^{-1}\equiv v_i$
    \item If $\sigma_{i}=-1,\sigma_{i+1}=1$, then $wu_iw^{-1}\equiv v_i$
    \item If $\sigma_i=\sigma_{i+1}= -1$, then $wu_i\varepsilon^{-1}w^{-1}\equiv v_i$.
\end{itemize}  
Also, recall that $\langle X,q\rangle_{M(G)}$ and $\langle \Theta\rangle_{M(G)}$ are free groups, so we may speak of freely reduced elements thereof.

We proceed with two lemmas and their corollaries:
\begin{lemma}\label{conjbythetas}
    For any $\tau_0\in \langle\Theta\rangle_{M(G)}$, there exist words  $w$ and $\varepsilon$ over $X$ such that $\varepsilon=_G1$ and $$\tau_0 q\tau_0^{-1}=_{M(G)}w^{-1}qw\varepsilon.$$ Moreover, $||w||\leq ||\tau_0||$.
\end{lemma}
\begin{proof}
    Every element of $\Theta^{\pm1}$ can be written $\theta_\alpha^{\delta}$ for some $\delta\in \pm1$ and $\alpha\in X\cup R$. We have $$\theta_\alpha^{\delta}x\theta_{\alpha}^{-\delta}=_{M(G)}x$$ for any $x\in X$ and $\delta\in \{\pm1\}$. Also, if $\alpha=x\in X$, then by the second defining relation for $M(G)$, $$\theta_\alpha^\delta q\theta_\alpha^{-\delta}=_{M(G)}x^{-\delta}qx^\delta,$$ while if $\alpha=r\in R,$ then the third defining relation gives $$\theta_\alpha^\delta q\theta_\alpha^{-\delta}=_{M(G)}qr^\delta.$$ We know $\tau_0$ is a product of elements of $\Theta^{\pm1}$, so $\tau q\tau^{-1}=_{M(G)}w^{-1}qw',$ where $w,w'$ are words on $X$ such that $w'$ is identical to $w$ except that elements of $R^{\pm1}$ inserted according to the $\theta_r$'s in $\tau_0$. Also, every letter of $w$ corresponds to a $\theta_x$-letter in $\tau_0,$ so $||w||\leq ||\tau_0||$.

    Now, $r=_G1$ for all $r\in R^{\pm1},$ so $w'=_Gw.$ But $G$ is a finitely presented group, so as words we must have $w'\equiv w\varepsilon$ for some $\varepsilon$ in the normal closure of $R$ in $F(X).$ Equivalently, we have $w'\equiv w\varepsilon$ for some $\varepsilon=_G1$, so $w^{-1}qw'\equiv w^{-1}qw\varepsilon$ as desired.

\end{proof} 
This lemma has two corollaries, which we now prove.

\begin{corollary}\label{invconjbythetas}
    For $\tau_0\in \langle \Theta\rangle_{M(G)}$, there exist words  $w$ and $\varepsilon$ over $X$ such that $\varepsilon=_G1$ and $$\tau_0 q^{-1}\tau_0^{-1}=_{M(G)}(w\varepsilon)^{-1}q^{-1}w.$$ 
\end{corollary}
\begin{proof}
    This follows by inverting the equation in Lemma \ref{conjbythetas}.
\end{proof}
\begin{corollary}\label{thetaspreserveXqs}
    If $w\in \langle X,q\rangle_{M(G)}$ and $\tau_0\in \langle\Theta\rangle_{M(G)}$, then we have $\tau_0 \alpha\tau_0^{-1}\in~\langle X,q\rangle_{M(G)}$.
\end{corollary}
\begin{proof}
    Since $\tau_0 x^{\pm1} \tau_0^{-1}=_{M(G)}x^{\pm1}\in \langle X,q\rangle_{M(G)} $ by definition, and since $\tau_0 q^{\pm1}\tau_0^{-1}\in~\langle X,q\rangle_{M(G)}$ by the above Lemma and Corollary, we also have that $w\in \langle X,q\rangle_{M(G)}$ implies $\tau_0w\tau_0^{-1}\in \langle X,q\rangle_{M(G)}$.
\end{proof}
\begin{remark} Note that Corollary \ref{thetaspreserveXqs} also follows from the discussion in Section~\ref{stableletterthetaHNN}.\end{remark}
Next, we show two more lemmas that will be using in proving this section's main result.
\begin{lemma}\label{wordstoconjbythetas}
    For any word $w$ on $X$, there exists $\tau_w\in \langle \Theta\rangle_{M(G)}$ such that $\tau_w q\tau_w^{-1}=w^{-1}qw,\ ||\tau_w||=||w||$, and $w\tau_w\in K_{-1}$. Also, for any word $\varepsilon$ with $\varepsilon=_G1$, there exists a $\tau_\varepsilon\in\langle \Theta\rangle$ such that $\tau_\varepsilon q\tau_\varepsilon=q\varepsilon$, $|\tau_\varepsilon|=\lambda(\varepsilon)$, and $\tau_\varepsilon\in K_{-1}$.
\end{lemma}

\begin{proof} We may write $w=x_1^{\delta_1}x_2^{\delta_2}...x_m^{\delta_m}$ for some $x_i\in X,\delta_i\in \{\pm1\}$. A simple computation gives $\tau_w=\theta_{x_1}^{\delta_1}\theta_{x_2}^{\delta_2}...\theta_{x_m}^{\delta_m}$ that is the desired word. Also,  $|\tau_w|=|w|$ by construction. Moreover, by the relations $\theta_\alpha x=x\theta_\alpha$, we see $$w\tau_w=_{M(G)}x_1^{\delta_1}\theta_{x_1}^{\delta_1}x_2^{\delta_2}\theta_{x_1}^{\delta_2}...x_m^{\delta_m}\theta_{x_m}^{\delta_m}$$ $$=_{M(G)}(x_1\theta_{x_1})^{\delta_1}(x_2\theta_{x_2})^{\delta_2}...(x_m\theta_{x_m})^{\delta_m}$$ which is in $K_{-1}$. For the next stage of our proof, note that $\tau_{w^{-1}}=\tau_w^{-1}$.

For the second part, recall that $\varepsilon=_G1$ implies that $\varepsilon$ can be written as a word $\prod_{i=1}^k w_i r_i^{\delta'_i} w_i^{-1}$ for words $w_i$ on $X$, $r_i\in R$, and $\delta'_i\in~\{\pm1\}$. Suppose without loss of generality that this word minimizes the value of $f$ (defined in Section \ref{lamlength}) over all words freely equal to $\varepsilon$, and define $\tau_\varepsilon=\prod_{i=1}^k \tau_{w_i} \theta_{r_i}^{\delta'_i} \tau_{w_i^{-1}}$. Then we see $$\tau_\varepsilon q\tau_{\varepsilon}^{-1}=_{M(G)}\left(\prod_{i=1}^k w_i  w_i^{-1}\right)^{-1} q \left(\prod_{i=1}^k w_i r_i^{\delta'_i} w_i^{-1}\right) =_{M(G)} q \prod_{i=1}^k w_i r_i^{\delta'_i} w_i^{-1}=_{M(G)}q\varepsilon$$ as desired. By the definition of $f$ and $\lambda$, $|\tau_\varepsilon|=\lambda(\varepsilon)$.  Lastly, we have $\tau_{w^{-1}}=\tau_w^{-1}$ and $\theta_{r_i},w_i\tau_{w_i}\in K_{-1}$ for all $i$, so $$\tau_{\varepsilon}=_{M(G)}\prod_{i=1}^k \tau_{w_i} \theta_{r_i}^{\delta'_i} \tau_{w_i^{-1}}=_{M(G)}\prod_{i=1}^k w_i\tau_{w_i} \theta_{r_i}^{\delta'_i} \tau_{w_i^{-1}}w_{i}^{-1}=\prod_{i=1}^k w_i\tau_{w_i} \theta_{r_i}^{\delta'_i} \tau_{w_i}^{-1}w_{i}^{-1}$$is in $K_{-1}$ as well, so we are done.
\end{proof}
\begin{lemma}\label{sameqsafterconjugation}
Suppose $x,y\in \langle X,q\rangle_{M(G)}$ are represented by cyclically and freely reduced words, and $x\approx y$. Then $x$ and $y$ have the same $q$-letters in the same order. More precisely, if 
    $$x=u_0q^{\sigma_1}u_1q^{\sigma_2}...q^{\sigma_k}u_k$$ and $$y=v_0q^{\varepsilon_1}v_1q^{\varepsilon_2}...q^{\varepsilon_\ell}v_\ell,$$ where both of these words are and cyclically freely reduced, then $x\approx y$ implies $k=\ell$ and $\sigma_i=\varepsilon_i$ for $i=1,...,k$.
    \end{lemma}
    \begin{proof}
        Let $\gamma$ be the conjugator taking $x$ to $y$, with $\gamma$ not containing any $q$-letters. Then, $\gamma$ can be written $z\mu$ for some $z\in \langle X\rangle_{M(G)}$, $\mu\in \langle \Theta\rangle_{M(G)}$. It suffices to show that conjugation by $\mu$ does not create any free cancelations of the $q$'s. By Lemma \ref{conjbythetas} and its corollaries, there exist $w,\varepsilon\in \langle X\rangle_{M(G)}$ with $\varepsilon=_G1$ such that $\mu q\mu^{-1}=_{M(G)}w^{-1}qw\varepsilon$ and $\mu q^{-1}\mu^{-1}=_{M(G)}(w\varepsilon)^{-1}q^{-1}w$. Conjugation by $\mu$ thus inserts $w^{\pm1}$ and $(w\varepsilon)^{\pm1}$ into $x$ according to the $\sigma_i$'s. If there is any cancellation among the $q$-letters, we must have the following situation (or its inverse) $$\mu qu_iq^{-1}\mu^{-1}=_{M(G)}w_1qq^{-1}w_2=_{M(G)}1$$ for some words $w_1,w_2$ on $x$. In this case, we see $$\mu qu_iq^{-1}\mu^{-1}=_{M(G)}w^{-1}q w\varepsilon u_i (w\varepsilon)^{-1}q^{-1}w.$$ In order for the $q$-letters to cancel, we must have that the word $w\varepsilon u_i (w\varepsilon)^{-1}$ is trivial. But this implies $u_i$ is trivial, so our word for $x$ was not freely reduced to begin with. This gives a contradiction, so we are done.\end{proof}

Now we prove the main result of this section.

\begin{theorem}\label{conjwithoutthetas}
    Fix $n\geq 1$, let $u_1,...,u_k,v_1,...,v_k\in \langle X\rangle_{M(G)}$ be reduced words, and let  $\sigma=(\sigma_1,...,\sigma_k)$ for $\sigma_1,...,\sigma_k\in \{\pm1\}$. Suppose $\sigma_i=-\sigma_{i+1}$ implies $u_i\not\equiv 1\not\equiv v_i$. We have $$q^{\sigma_1}u_1q^{\sigma_2}u_2 \cdots q^{\sigma_k}u_k\approx q^{\sigma_1}v_1q^{\sigma_2}v_2 \cdots q^{\sigma_k}v_k$$ if and only if  $(u_1,...,u_k)\overset{\sigma}{\sim} (v_1,...,v_k)$.
\end{theorem}

\begin{proof}

Suppose $$q^{\sigma_1}u_1q^{\sigma_2}u_2 \cdots q^{\sigma_k}u_k\approx q^{\sigma_1}v_1q^{\sigma_2}v_2 \cdots q^{\sigma_k}v_k$$ via the conjugator $\gamma\in \langle X\cup\Theta\rangle_{M(G)}$. First, observe that $q^{\sigma_1}u_1q^{\sigma_2}u_2 \cdots q^{\sigma_k}u_k$ and $q^{\sigma_k}u_k\approx q^{\sigma_1}v_1q^{\sigma_2}v_2 \cdots q^{\sigma_k}v_k$ are cyclically reduced since $\langle X,q\rangle_{M(G)}$ is free and $\sigma_i=-\sigma_{i+1}$ implies $u_i\not\equiv 1\not\equiv v_i$. We may write $\gamma$ as $z\mu$, where $z\in \langle X\rangle_{M(G)}$, $\mu\in \langle \Theta\rangle_{M(G)}$. By Lemma \ref{conjbythetas} and Corollary \ref{invconjbythetas}, there exist some words $w,\varepsilon$ on $X$, with $\varepsilon=_G1$, such that $\mu q\mu^{-1}=_{M(G)} w^{-1}qw\varepsilon$ and $\mu q^{-1}\mu^{-1}=_{M(G)}(w\varepsilon)^{-1}q^{-1}w$. Also, $\mu u_i\mu^{-1}=_{M(G)}u_i$ since every element of $\langle \Theta^{\pm1}\rangle_{M(G)}$ commutes with each $u_i$. Hence, we see $$\gamma q^{\sigma_1}u_1q^{\sigma_2}u_2 \cdots q^{\sigma_k}u_k\gamma^{-1}=_{M(G)}zu_0'q^{\sigma_1}u_1'q^{\sigma_2}u_2' \cdots q^{\sigma_k}u_k'z^{-1}$$ where $u_0'$ is $w^{-1}$ if $\sigma_1=1$ and $(w\varepsilon)^{-1}$ if $\sigma_1=-1$, and the rest of the $u_i$'s are as follows. By the above, for $i=1,...,k-1$, $u_i'$ is of the following form:

\begin{itemize}
    \item If $\sigma_i=\sigma_{i+1}=1$, then $u_i'=_{M(G)}w\varepsilon u_i w^{-1}$
    \item If $\sigma_i=1,\sigma_{i+1}=-1$, then $u_i'=_{M(G)}w\varepsilon u_i \varepsilon^{-1}w^{-1}$
    \item If $\sigma_{i}=-1,\sigma_{i+1}=1$, then $u_i'=_{M(G)}wu_iw^{-1}$
    \item If $\sigma_i=-1,\sigma_{i+1}= -1$, then $u_i'=_{M(G)}wu_i\varepsilon^{-1}w^{-1}.$
\end{itemize} 
Now, by the argument in Lemma \ref{sameqsafterconjugation} (which we may apply since our words are cyclically reduced), there is no cancellation among the $q$-letters of the word $zu_0'q^{\sigma_1}u_1'q^{\sigma_2}u_2' \cdots q^{\sigma_k}u_k'z^{-1}$. Thus, since $\langle X,q\rangle_{M(G)}$ is free, the subwords between them must be pairwise equal to the corresponding subwords in the word $q^{\sigma_1}v_1q^{\sigma_2}v_2 \cdots q^{\sigma_k}v_k$. That is,  $zu_0'=_{M(G)}1$ and $u_i'=_{M(G)}v_i$ for $i=1,...,k-1$, as well as $u_k'z^{-1}=_{M(G)}v_k$. Moreover, each of these subwords is an element of $\langle X\rangle_{M(G)}$, which is free, so these equalities are free. Letting $\sigma=(\sigma_1,...,\sigma_k)$, the only thing remaining to show is that the four implications hold for $u_k$ and $v_k$. We will prove that the first implication is true, with the rest following by analogous arguments.

Recall our convention that $\sigma_{k+1}=\sigma_1$. Suppose $\sigma_k=\sigma_{k+1}=1$. Then, $\sigma_1=1$ so $u_0'$ is $w^{-1}$. Since $zu_0'\equiv 1$ by the above, $z\equiv w$. Thus, $u_k'\equiv w\varepsilon u_k w^{-1}$. We asserted above that $u_k'z^{-1}\equiv v_n$, hence $w\varepsilon u_k w^{-1}\equiv v_k$, which is the desired equality.

Now we show the converse direction. Let $(u_1,...,u_k),(v_1,...,v_k),\sigma=(\sigma_1,...,\sigma_k)$ be given as above, and suppose that $(u_1,...,u_k)\overset{\sigma}{\sim}(v_1,...,v_k)$ (this is possible because of the first biconditional). Let $w$ and $\varepsilon$ be the words given in the definition of $\overset{\sigma}{\sim}$, meaning $\varepsilon=_G1$ and \begin{itemize}
    \item If $\sigma_i=\sigma_{i+1}=1$, then $w\varepsilon u_i w^{-1}\equiv v_i$
    \item If $\sigma_i=1,\sigma_{i+1}=-1$, then $w\varepsilon u_i \varepsilon^{-1}w^{-1}\equiv v_i$
    \item If $\sigma_{i}=-1,\sigma_{i+1}=1$, then $wu_iw^{-1}\equiv v_i$
    \item If $\sigma_i=-1,\sigma_{i+1}= -1$, then $wu_i\varepsilon^{-1}w^{-1}\equiv v_i$.
\end{itemize}  
for $i=1,...,k$. If $\sigma_1=1,$ define (using Lemma \ref{wordstoconjbythetas}) $\gamma=w\tau_w\tau_\varepsilon$, and if $\sigma_1=-1$ define $\gamma=w\varepsilon\tau_w\tau_\varepsilon$. By the above relations, we compute that $$\gamma q^{\sigma_1}u_1q^{\sigma_2}u_2 \cdots q^{\sigma_k}u_k\gamma^{-1}=_{M(G)} q^{\sigma_1}v_1q^{\sigma_2}v_2 \cdots q^{\sigma_k}v_k.$$ This completes the proof.
    \end{proof}

\section{Converse Direction of Theorem A}\label{mainalgorithm}

First, a lemma.
\begin{lemma}\label{notypeC}Let $t=\max\{|r|:r\in R\}+1$. There exists a constant $C$ such that, for every conjugate pair $x\sim_{M(G)}y$ where the conjugacy is witnessed via a diagram with no $q$-corridors of type C, there exists a conjugator $g\in M(G)$ with length $|g|\leq 3t^2 C(|x|+|y|)$.
\end{lemma}
\begin{proof}
    Suppose $x\sim_{M(G)}y$, but there exists no diagram witnessing $x\sim_{M(G)}y$ with $q$-corridors of type C. First, if there is a subword of of the form $qwq^{-1}$, where $w\in K_1$, we replace it with the corresponding word $w'$ on the generators of $K_{-1}$, and replace similarly all words of the form $q^{-1}wq$, where $w\in K_{-1}$, with a word on the generators of $K_1$. This increases the length of $x$ and $y$ by at most a factor of $t$. Now, we fix some diagram witnessing $x\sim_{M(G)}y$. Applying the $q$-conjugations eliminates all $q$-corridors of type A in this diagram, so we may assume the diagram only contains $q$-corridors of types B and D. The word along outside boundary of a corridor $Q$ of type B is an element of $\langle X\cup\Theta\rangle_{M(G)}$. Since it is also the boundary of a disk, it is trivial in $M(G)$, and hence trivial in $\langle X\cup\Theta\rangle_{M(G)}$. Thus, we can fill in this disk using just cells corresponding to the commutation relations $x\theta_\alpha=\theta_\alpha x$, thereby eliminating the $q$-corridor. Since $q$-corridors cannot cross, we have not affected any other $q$-corridors, except for those completely enclosed by $Q$, which we have also eliminated. 
    
    Let us do this for every $q$-corridor of type B. This gives us a diagram with only $q$-corridors of type D. Suppose there are two $q$-corridors of type D. Each of their cells has two $\theta_\alpha$ edges, which must be part of a $\theta$-corridor. By the same argument as in the beginning of the proof Lemma \ref{thcorwrap}, which does not use the assumption that no $\theta$'s are on the boundary of the diagram, a $\theta$-corridor cannot intersect either $q$-corridor twice. By inspection, we see that a $\theta$-corridor of types A,B, or D cannot intersect a $q$-corridor of type C precisely one time. Thus, any $\theta$-corridor intersecting one of them must be of type C, so we have following situation in Figure \ref{fig:qandtheta_crossing}.
    \begin{figure}
        \centering
        \includegraphics[scale=.35]{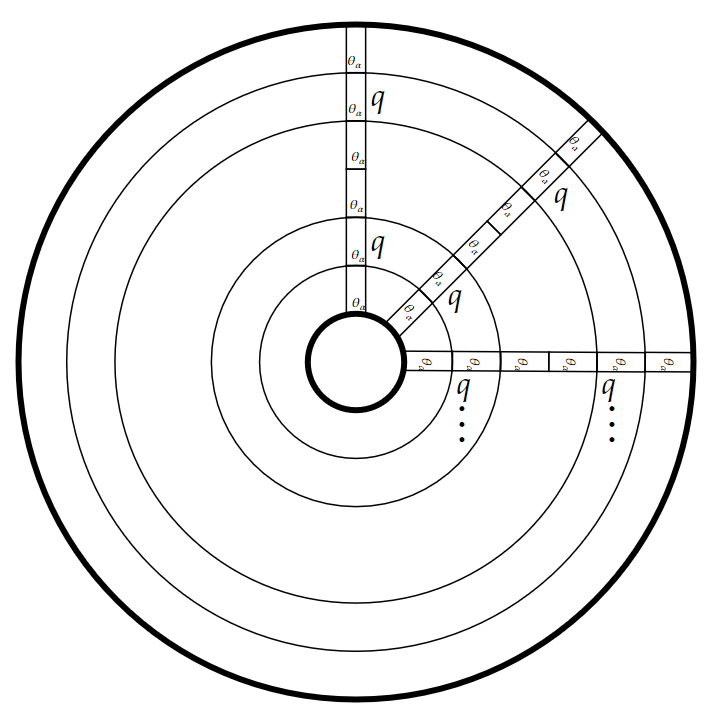}
        \caption{$\theta$-corridors crossing $q$-corridors}
        \label{fig:qandtheta_crossing}
    \end{figure}

    Since no two $\theta$-corridors can cross, this implies that the $\theta$-edges of both $q$-corridors must be the same, and in the same order. The words along the boundary of a $q$-corridor are uniquely determined by its $\theta$-edges (via $\varphi^{-1},$ as defined in Section \ref{stabletterqHNN}) hence both $q$-corridors have the same two boundary words, at least up to cyclic permutation. Both cases are depicted in Figure \ref{fig:sameq_directions}.

\begin{figure}
    \centering
$\begin{matrix}
\tikzset{every picture/.style={line width=0.75pt}} 

\begin{tikzpicture}[x=0.5pt,y=0.5pt,yscale=-.5,xscale=.5]

\draw  [line width=3]  (50,250) .. controls (50,111.93) and (161.93,0) .. (300,0) .. controls (438.07,0) and (550,111.93) .. (550,250) .. controls (550,388.07) and (438.07,500) .. (300,500) .. controls (161.93,500) and (50,388.07) .. (50,250) -- cycle ;
\draw  [line width=3]  (264.29,250) .. controls (264.29,230.28) and (280.28,214.29) .. (300,214.29) .. controls (319.72,214.29) and (335.71,230.28) .. (335.71,250) .. controls (335.71,269.72) and (319.72,285.71) .. (300,285.71) .. controls (280.28,285.71) and (264.29,269.72) .. (264.29,250) -- cycle ;
\draw   (228.57,250) .. controls (228.57,210.55) and (260.55,178.57) .. (300,178.57) .. controls (339.45,178.57) and (371.43,210.55) .. (371.43,250) .. controls (371.43,289.45) and (339.45,321.43) .. (300,321.43) .. controls (260.55,321.43) and (228.57,289.45) .. (228.57,250) -- cycle ;
\draw  [dash pattern={on 0.84pt off 2.51pt}] (192.86,250) .. controls (192.86,190.83) and (240.83,142.86) .. (300,142.86) .. controls (359.17,142.86) and (407.14,190.83) .. (407.14,250) .. controls (407.14,309.17) and (359.17,357.14) .. (300,357.14) .. controls (240.83,357.14) and (192.86,309.17) .. (192.86,250) -- cycle ;
\draw   (121.43,250) .. controls (121.43,151.38) and (201.38,71.43) .. (300,71.43) .. controls (398.62,71.43) and (478.57,151.38) .. (478.57,250) .. controls (478.57,348.62) and (398.62,428.57) .. (300,428.57) .. controls (201.38,428.57) and (121.43,348.62) .. (121.43,250) -- cycle ;
\draw  [dash pattern={on 0.84pt off 2.51pt}] (85.71,250) .. controls (85.71,131.65) and (181.65,35.71) .. (300,35.71) .. controls (418.35,35.71) and (514.29,131.65) .. (514.29,250) .. controls (514.29,368.35) and (418.35,464.29) .. (300,464.29) .. controls (181.65,464.29) and (85.71,368.35) .. (85.71,250) -- cycle ;
\draw    (300,35.71) -- (300,69.43) ;
\draw [shift={(300,71.43)}, rotate = 270] [color={rgb, 255:red, 0; green, 0; blue, 0 }  ][line width=0.75]    (10.93,-3.29) .. controls (6.95,-1.4) and (3.31,-0.3) .. (0,0) .. controls (3.31,0.3) and (6.95,1.4) .. (10.93,3.29)   ;
\draw    (300,144.29) -- (300,176.57) ;
\draw [shift={(300,178.57)}, rotate = 270] [color={rgb, 255:red, 0; green, 0; blue, 0 }  ][line width=0.75]    (10.93,-3.29) .. controls (6.95,-1.4) and (3.31,-0.3) .. (0,0) .. controls (3.31,0.3) and (6.95,1.4) .. (10.93,3.29)   ;

\draw (307,42.4) node [anchor=north west][inner sep=0.75pt]    {$q$};
\draw (307,152.4) node [anchor=north west][inner sep=0.75pt]    {$q$};
\end{tikzpicture}
&&
\tikzset{every picture/.style={line width=0.75pt}} 

\begin{tikzpicture}[x=0.5pt,y=0.5pt,yscale=-.5,xscale=.5]

\draw  [line width=3]  (50,250) .. controls (50,111.93) and (161.93,0) .. (300,0) .. controls (438.07,0) and (550,111.93) .. (550,250) .. controls (550,388.07) and (438.07,500) .. (300,500) .. controls (161.93,500) and (50,388.07) .. (50,250) -- cycle ;
\draw  [line width=3]  (264.29,250) .. controls (264.29,230.28) and (280.28,214.29) .. (300,214.29) .. controls (319.72,214.29) and (335.71,230.28) .. (335.71,250) .. controls (335.71,269.72) and (319.72,285.71) .. (300,285.71) .. controls (280.28,285.71) and (264.29,269.72) .. (264.29,250) -- cycle ;
\draw  [dash pattern={on 0.84pt off 2.51pt}] (228.57,250) .. controls (228.57,210.55) and (260.55,178.57) .. (300,178.57) .. controls (339.45,178.57) and (371.43,210.55) .. (371.43,250) .. controls (371.43,289.45) and (339.45,321.43) .. (300,321.43) .. controls (260.55,321.43) and (228.57,289.45) .. (228.57,250) -- cycle ;
\draw   (192.86,250) .. controls (192.86,190.83) and (240.83,142.86) .. (300,142.86) .. controls (359.17,142.86) and (407.14,190.83) .. (407.14,250) .. controls (407.14,309.17) and (359.17,357.14) .. (300,357.14) .. controls (240.83,357.14) and (192.86,309.17) .. (192.86,250) -- cycle ;
\draw   (121.43,250) .. controls (121.43,151.38) and (201.38,71.43) .. (300,71.43) .. controls (398.62,71.43) and (478.57,151.38) .. (478.57,250) .. controls (478.57,348.62) and (398.62,428.57) .. (300,428.57) .. controls (201.38,428.57) and (121.43,348.62) .. (121.43,250) -- cycle ;
\draw  [dash pattern={on 0.84pt off 2.51pt}] (85.71,250) .. controls (85.71,131.65) and (181.65,35.71) .. (300,35.71) .. controls (418.35,35.71) and (514.29,131.65) .. (514.29,250) .. controls (514.29,368.35) and (418.35,464.29) .. (300,464.29) .. controls (181.65,464.29) and (85.71,368.35) .. (85.71,250) -- cycle ;
\draw    (300,34.29) -- (300,68) ;
\draw [shift={(300,70)}, rotate = 270] [color={rgb, 255:red, 0; green, 0; blue, 0 }  ][line width=0.75]    (10.93,-3.29) .. controls (6.95,-1.4) and (3.31,-0.3) .. (0,0) .. controls (3.31,0.3) and (6.95,1.4) .. (10.93,3.29)   ;
\draw    (300,178.57) -- (300,144.86) ;
\draw [shift={(300,142.86)}, rotate = 90] [color={rgb, 255:red, 0; green, 0; blue, 0 }  ][line width=0.75]    (10.93,-3.29) .. controls (6.95,-1.4) and (3.31,-0.3) .. (0,0) .. controls (3.31,0.3) and (6.95,1.4) .. (10.93,3.29)   ;

\draw (307,42.4) node [anchor=north west][inner sep=0.75pt]    {$q$};
\draw (307,152.4) node [anchor=north west][inner sep=0.75pt]    {$q$};

\end{tikzpicture}
\end{matrix}$
    \caption{$q$-corridors of type D with $q$-edges facing in the same and opposite direction}
    \label{fig:sameq_directions}
\end{figure}
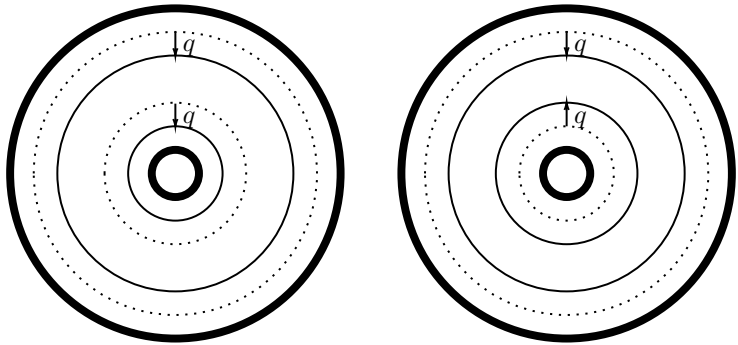

    Because the words along the dotted edges are the same, we can remove the part of the diagram between them, and glue along the dotted edges, thereby obtaining a new annular diagram witnessing $x\sim_{M(G)}y.$ This reduces the number of type D $q$-corridors by at least one. Let us perform this procedure as many times as possible, so that there is at most one $q$-corridor of type D.

     The sum of the lengths of the boundary components here is at most $t(|x|+|y|)$, since eliminating type A $q$-corridors increases length by a factor of $t$, and all of our other manipulations have no effect on the boundary. We conclude this proof by checking two cases.
    \paragraph{Case 1.} Suppose, after applying all the above to the diagram, there is no $q$-corridor of type D. Then, there are no $q$-cells at all on the diagram, hence $x$ and $y$ are elements of $\langle X\cup \Theta\rangle_{M(G)}$, and are conjugate via an element of $\langle X\cup \Theta\rangle_{M(G)}$ as well. This group is a direct product of the free groups $\langle X\rangle_{M(G)}$ and $\langle \Theta\rangle_{M(G)}$ and therefore has linear conjugator length function. That is, there exists a constant $C$ such that, for any  $x$ and $y$ that are conjugate in $\langle X\cup \Theta\rangle_{M(G)}$, there is a conjugator $h\in \langle X\cup \Theta\rangle_{M(G)}$ whose length in $\langle X\cup \Theta\rangle_{M(G)}$ is at most $ Ct(|x|+|y|)$.  Since every generator of $\langle X\cup \Theta\rangle_{M(G)}$ is a generator of $M(G)$, the length of $h$ in $\langle X\cup \Theta\rangle_{M(G)}$ is at least $|h|$, the length of $h$ in $M(G).$ Therefore, we have shown that there is a conjugator in $M(G)$ taking $x$ to $y$ with length at most $Ct(|x|+|y|)<3t^2 C(|x|+|y|)$.

    \paragraph{Case 2.} Suppose there is one $q$-corridor $Q$ of type D. By inspection of the cells, the lengths of the words along both boundaries of a $q$-corridor of type D are bounded above by $ts$, where $s$ is the number of $\theta$-edges (which is the same on both boundaries). Since these $\theta$-edges lie on corridors of type C, $s\leq \min\{|x|,|y|\}$. Now, the word $w_1$ on the interior boundary of $Q$ is an element of $\langle X\cup \Theta\rangle_{M(G)}$, and is conjugate to the word along the interior of the entire diagram, which we can assume without loss of generality is $x$. The only $q$-edges in the diagram are edges of $Q$, since $Q$ is the only $q$-corridor, so $w_1$ and $x$ are conjugate in $\langle X\cup \Theta\rangle_{M(G)}.$ By the same argument as Case 1, there exists a conjugator $h_1\in M(G)$ taking $x$ to $w_1$ with $$|h_1|\leq Ct(|x|+|w_1|)\leq Ct(|x|+ts)\leq Ct(|x|+t |y|)\leq  t^2C(|x|+|y|).$$ Similarly, the outer word $w_2$ of $Q$ is conjugate to $y$ via a conjugator $h_2$ of length $$|h_2|\leq Ct(|y|+|w_2|)\leq Ct(|y|+t |x|)\leq t^2 C(|x|+|y|).$$ Since $w_1$ and $w_2$ are words along the same side of a $q$-corridor, they are conjugate by $q^{\pm1}$. Thus, $x$ and $y$ are conjugate by $h_2q^{\pm1}h_1$, which has length at most $2t^2 C(|x|+|y|)+1\leq 3t^2 C(|x|+|y|)$.
    \end{proof}

\begin{remark}\label{conjviaCcorridor}
    If $x$ and $y$ are conjugate via a diagram with no $q$-corridors of type C, then this lemma gives an upper bound for $c(x,y)$. If they are conjugate via a diagram \textit{with} such a corridor, then reading off the word along that corridor gives a conjugator in $\langle X\cup\Theta\rangle_{M(G)}$ taking a cyclic permutation $x'$ of $x$ to a cyclic permutation $y'$ of $y$. In other words, there exist cyclic permutations $x'\sim_{M(G)}x$ and $y'\sim_{M(G)}y$ such that $x'\approx y'$.
\end{remark}

For the final ingredients of our algorithm, we borrow some definitions from \cite{BMR}. Namely, to each $x\in M(G)$ the authors associate a $c(x)\sim_{M(G)}x$, which they call \textit{cyclically reduced} and which is computable from $x$. If $c(x)\not\in \langle X,q\rangle_{M(G)}$, they then call $c(x)$ \textit{weakly regular}. Note that $\langle X,q\rangle_{M(G)}$ is a normal subgroup of $M(G)$, so if $c(x)$ is weakly regular and $c(y)$ is not, then $c(x)\not\sim_{M(G)}c(y)$, and hence $x\not\sim_{M(G)}y$.

Now we give our algorithm, assuming the existence of an algorithm deciding $\textbf{u}\overset{\sigma}{\sim}\textbf{v}$. Note that a solution to the word problem is given in \cite{Miller71}, which we will use repeatedly in our algorithm. We will also assume that we can always put an element in the form $\alpha\tau$, for $\alpha\in \langle X,q\rangle_{M(G)}$ and $\tau\in \langle \Theta\rangle_{M(G)}$, using polynomial space.

\paragraph{Algorithm:}We are given some $x,y\in M(G)$. We may assume they are already in cyclically reduced form.

\begin{enumerate}
\item Iterate over all $\gamma$ with $|\gamma|\leq 3Ct^2(|x|+|y|)$ and check if $\gamma x\gamma^{-1}=_{M(G)}y$. If such a $\gamma$ is found, conclude $x\sim_{M(G)}y$.
\item Otherwise, check if $x$ and $y$ are weakly regular. \item If only one is, conclude $x\not\sim_{M(G)}y$. 

\item If they both are weakly regular, use \cite[Theorem 4.9]{BMR} to check $x\sim_{M(G)}~y$ in cubic time.
\item Lastly, if both are not weakly regular, they have no $\theta$-letters. So, for every cyclic permutation $x'$ and $y'$ of $x$ and $y$ respectively, write $x'=u_0q^{\sigma_1}u_1q^{\sigma_1}u_2...q^{\sigma_k}w_k$, $y'=v_0q^{\varepsilon_1}v_1q^{\varepsilon_1}v_2...q^{\varepsilon_m}v_m$ and freely reduce both words.
\begin{enumerate}
\item If $m\neq k$ or $\varepsilon_i\neq \sigma_i$ for some $i\leq m=k$, conclude $x'\not\approx y'$ by Lemma \ref{sameqsafterconjugation}. \item Otherwise, compute the tuples $\textbf{u}=(u_1,...,u_ku_0),\textbf{v}=(v_1,...,v_kv_0),\sigma=(\sigma_1,...,\sigma_k)$. If $\textbf{u}\overset{\sigma}{\sim}\textbf{v}$, conclude $x'\approx y'$ by Theorem \ref{conjwithoutthetas} (both $x'$ and $y'$ are reduced, so the hypotheses of this theorem are satisfied). Chaining these relations together gives $x\sim_{M(G)}x'\approx y'\sim_{M(G)}y$, hence conclude $x\sim_{M(G)}y$.
\end{enumerate}
\item If all cyclic permutations fail, conclude $x\not\sim_{M(G)}y$.

\end{enumerate}

\begin{proof}[Proof of Correctness]
If $x$ and $y$ are conjugate via a diagram with no $q$-corridors of type C, the first bullet point will find their conjugator by Lemma \ref{notypeC}. If they are conjugate by a diagram \textit{with} such a $q$-corridor, we see that either they are both weakly regular, or neither are. In the first case, \cite[Theorem 4.9]{BMR} gives an algorithm to determine if $x\sim_{M(G)}y$.

In the second case, by Remark \ref{conjviaCcorridor} there exist cyclic conjugates $x'$ and $y'$ such that $x'\approx y'$. Theorem \ref{conjwithoutthetas} implies that $u_0^{-1}x'u_0\approx v_0^{-1}y'v_0$ if and only if $\textbf{u}\overset{\sigma}{\sim}\textbf{v}$. Since $x'\approx u_0^{-1}x'u_0$, $ v_0^{-1}y'v_0 \approx y'$ this implies $x'\approx y'$ if and only if $\textbf{u}\overset{\sigma}{\sim}\textbf{v}$. Combining this with Lemma \ref{sameqsafterconjugation} completes our proof.
\end{proof}

\begin{corollary}
    If $G$ is a finite group, then the conjugacy problem for $M(G)$ is in $\mathsf{PSPACE}$.
\end{corollary}
\begin{proof}
    Step (1.)\ can be decided in polynomal-space, reusing the same space for each iteration. Steps (2.)\ through (4.)\ likewise use only polynomail space. Note that step (5.)\ is the only remaining step containing any computation, that it iterates only polynomially-many times, and that a cyclic permutation of a word can be computed in linear space. Step (5a.)\ merely reads words of polynomial (indeed linear) length and compares many integers of polynomial size, and hence takes polynomial space. For Step (5b.)\ the tuples $\textbf{u}$, $\textbf{v}$, and $\sigma$ can likewise be computed in polynomial space. Lastly, deciding if $\textbf{u}\overset{\sigma}{\sim}\textbf{v}$ is equivalent to solving a system of equations in the variables $w$ and $\varepsilon$, with the constraint that $\varepsilon\in \langle \langle R\rangle\rangle$. If $G$ is finite, then $\langle \langle R\rangle\rangle$ is finite index in $F(X)$ and therefore finitely generated. Solving a system of equations with such a constraint is given as Problem 9.25 of the Kurovka notebook \cite{mazurov2023unsolved}, which was both solved and shown to be in $\mathsf{PSPACE}$ in \cite{rational_constraints}. This completes our proof.
\end{proof}

\section{Proof of Theorem B}\label{WordprobG}
To begin this section, we prove two useful lemmas.

\begin{lemma}
Let $t=\max(\{|r|:r\in R\}\cup\{2\})$. Let $x=q^{\sigma_1}u_1q^{\sigma_2}u_2 \cdots q^{\sigma_k}u_k\rho$ and $y=q^{\sigma_1}v_1q^{\sigma_2}v_2 \cdots q^{\sigma_k}v_k\tau$ be reduced words (not necessarily reduced elements), where $u_i,v_i$ are words over $X^{\pm1}$, $\tau,\rho\in \langle \Theta\rangle_{M(G)}$, and $n\geq 1.$ Suppose $x\sim_{M(G)}y$. If the diagram witnessing a minimal length conjugator has no $q$-corridors of type A, and no $\theta$-corridors of types A or C, then there exists a cyclic conjugate $y'$ of $y$ such that  $c'(x,y')\leq t c(x,y)$. 
\end{lemma} 
\begin{proof}For any word $\alpha$ on the generators of $M(G),$ let $\#_\theta(\alpha)$ be the number of $\theta$-letters in $\alpha.$ Thus $\#_\theta(\theta_x x \theta_x)=2$, $\#_\theta(\theta_r^{-1} x)=1,$ $\#_\theta(\theta_x\theta^{-1}_x\theta_r)=3,$ and so on.
    
    Now, let $h$ be a minimal length conjugator taking $x$ to $y$, so $|h|=c(x,y)$. If we draw an annular diagram for $x$ and $y$, we know there is a $q$-corridor going from the start of $x$ to some point in $y$. Let $y'$ be the cyclic conjugate of $y$ that begins at this point. On the side of the boundary that starts at the same point of $h$, the word $\delta$ is an element of $K_{\pm1},$ and an examination of $K_{\pm1}$'s generators shows $|\delta|\leq t\#_\theta(\delta).$ Moreover, $\delta$ conjugates $x$ to $y'$, so $c'(x,y')\leq |\delta|$. The situation described is depicted in Figure \ref{fig:conjgator_vs_corridor}.

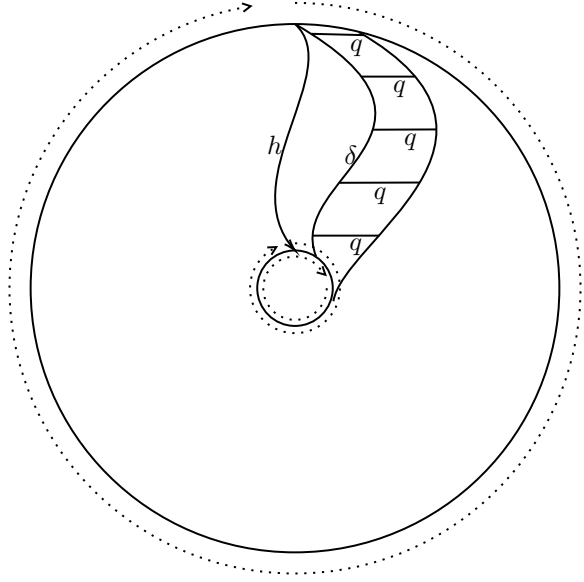
\begin{figure}
    \centering
\tikzset{every picture/.style={line width=0.75pt}} 

\begin{tikzpicture}[x=0.4pt,y=0.4pt,yscale=-1,xscale=1]

\draw  (50,350) .. controls (50,211.93) and (161.93,100) .. (300,100) .. controls (438.07,100) and (550,211.93) .. (550,350) .. controls (550,488.07) and (438.07,600) .. (300,600) .. controls (161.93,600) and (50,488.07) .. (50,350) -- cycle ;
\draw  (264.29,350) .. controls (264.29,330.28) and (280.28,314.29) .. (300,314.29) .. controls (319.72,314.29) and (335.71,330.28) .. (335.71,350) .. controls (335.71,369.72) and (319.72,385.71) .. (300,385.71) .. controls (280.28,385.71) and (264.29,369.72) .. (264.29,350) -- cycle ;
\draw    (300,100) .. controls (348.09,143.45) and (240.75,247.29) .. (299.11,313.29) ;
\draw [shift={(300,314.29)}, rotate = 227.55] [color={rgb, 255:red, 0; green, 0; blue, 0 }  ][line width=0.75]    (10.93,-3.29) .. controls (6.95,-1.4) and (3.31,-0.3) .. (0,0) .. controls (3.31,0.3) and (6.95,1.4) .. (10.93,3.29)   ;
\draw    (300,100) .. controls (474.33,201) and (289.67,245.67) .. (320,320) ;
\draw    (360.33,107.67) .. controls (534.67,208.67) and (345,305.67) .. (336.33,362.33) ;
\draw    (315.5,109.25) -- (365,110.5) ;
\draw    (363,149.75) -- (413.75,149.75) ;
\draw    (373.25,200) -- (433.75,199.75) ;
\draw    (342,250.25) -- (418,250) ;
\draw    (316.25,300.25) -- (379.75,300.25) ;
\draw  [draw opacity=0][dash pattern={on 0.84pt off 2.51pt}] (330,350) .. controls (330,350) and (330,350) .. (330,350) .. controls (330,366.57) and (316.57,380) .. (300,380) .. controls (283.43,380) and (270,366.57) .. (270,350) .. controls (270,333.43) and (283.43,320) .. (300,320) .. controls (311.1,320) and (320.8,326.03) .. (325.99,335) -- (300,350) -- cycle ; \draw  [dash pattern={on 0.84pt off 2.51pt}] (330,350) .. controls (330,350) and (330,350) .. (330,350) .. controls (330,366.57) and (316.57,380) .. (300,380) .. controls (283.43,380) and (270,366.57) .. (270,350) .. controls (270,333.43) and (283.43,320) .. (300,320) .. controls (311.1,320) and (320.8,326.03) .. (325.99,335) ;  
\draw   (326,329.24) -- (328.75,337.37) -- (320.51,334.99) ;
\draw  [draw opacity=0][dash pattern={on 0.84pt off 2.51pt}] (300,307.5) .. controls (323.47,307.5) and (342.5,326.53) .. (342.5,350) .. controls (342.5,373.47) and (323.47,392.5) .. (300,392.5) .. controls (276.53,392.5) and (257.5,373.47) .. (257.5,350) .. controls (257.5,334.27) and (266.05,320.54) .. (278.75,313.19) -- (300,350) -- cycle ; \draw  [dash pattern={on 0.84pt off 2.51pt}] (300,307.5) .. controls (323.47,307.5) and (342.5,326.53) .. (342.5,350) .. controls (342.5,373.47) and (323.47,392.5) .. (300,392.5) .. controls (276.53,392.5) and (257.5,373.47) .. (257.5,350) .. controls (257.5,334.27) and (266.05,320.54) .. (278.75,313.19) ;  
\draw   (273.73,311.43) -- (282.31,311.06) -- (277.71,318.31) ;
\draw  [draw opacity=0][dash pattern={on 0.84pt off 2.51pt}] (300,80) .. controls (300,80) and (300,80) .. (300,80) .. controls (449.12,80) and (570,200.88) .. (570,350) .. controls (570,499.12) and (449.12,620) .. (300,620) .. controls (150.88,620) and (30,499.12) .. (30,350) .. controls (30,216.88) and (126.34,106.26) .. (253.11,84.06) -- (300,350) -- cycle ; \draw  [dash pattern={on 0.84pt off 2.51pt}] (300,80) .. controls (300,80) and (300,80) .. (300,80) .. controls (449.12,80) and (570,200.88) .. (570,350) .. controls (570,499.12) and (449.12,620) .. (300,620) .. controls (150.88,620) and (30,499.12) .. (30,350) .. controls (30,216.88) and (126.34,106.26) .. (253.11,84.06) ;  
\draw   (249.88,80.51) -- (258.06,83.1) -- (251.26,88.34) ;

\draw (271,202.4) node [anchor=north west][inner sep=0.75pt]    {$h$};
\draw (344,212) node [anchor=north west][inner sep=0.75pt]    {$\delta $};
\draw (349.2,111.6) node [anchor=north west][inner sep=0.75pt]    {$q$};
\draw (390,150.8) node [anchor=north west][inner sep=0.75pt]    {$q$};
\draw (400.4,201.2) node [anchor=north west][inner sep=0.75pt]    {$q$};
\draw (371.6,251.6) node [anchor=north west][inner sep=0.75pt]    {$q$};
\draw (348.8,300.4) node [anchor=north west][inner sep=0.75pt]    {$q$};

\end{tikzpicture}

    \caption{Two conjugators, one of which is along a $q$-corridor}
    \label{fig:conjgator_vs_corridor}
\end{figure}

    Here, the innermost dotted loop represents the word representing $y'$, the middle one the word representing $y$ (note they are both read off along the same cycle of the diagram, and hence are cyclic conjugates), and the outer one the word representing $x$. The words $h$ and $\delta$ are also drawn.

Next we show $\#_\theta(\delta)\leq |h|$. Firstly, we assume that $\delta$ is reduced in $K_{\pm1},$ since otherwise we can shorten the $q$-spoke without affecting $|h|,$ as shown in Section \ref{Diag}. Then, every $\theta$-letter of $\delta$ must be in a $\theta$-corridor of type D. Our conjugator $h$ must cross each of these $\theta$-corridors at least once, and each crossing adds one to the length of $h$. Hence, $\#_\theta(\delta)\leq |h|$  which gives $$c'(x,y')\leq |\delta|\leq t\#_\theta(\delta),$$ as desired.
\end{proof}
\begin{lemma}\label{DandDprimerelationship} Let $t=\max(\{|r|:r\in R\}\cup\{2\})$. If $\sigma$ is any $k$-tuple over $\{\pm1\}$, then $$\frac{D_{k,\sigma}'(n)}{t}\leq D_{k,\sigma}(n)\leq D_{k,\sigma}'(n)$$ for all $n$. Also, $$\frac{D_0'(n)}{t}\leq D_0(n)\leq D_0'(n).$$
\end{lemma}
\begin{proof}
Let $\sigma=(\sigma_1,...,\sigma_k)$ be any $k$-tuple of $\{\pm1\}$, and $$x=q^{\sigma_1}u_1q^{\sigma_2}...q^{\sigma_k}u_k,$$ $$y=q^{\sigma_1}v_1q^{\sigma_2}...q^{\sigma_k}v_k$$ be such that $x\sim_{M(G)}y$ and $\sum_{i}||u_i||+||v_i||\leq n-2k$.  Since $x$ and $y$ have trivial $\theta$-factors, there are no $q$-corridors of type A in a diagram witnessing this conjugacy. Hence, every $q$-corridor in such a diagram is of type C. Thus, because $q$-corridors cannot cross, $q$-corridors of type C give a bijection between the $q$-letters on each boundary component. In particular, this bijection preserves order up to cyclic permutation. Since there are also no $\theta$-corridors of type A or C (since $x$ and $y$ contain no $\theta$-letters), the previous lemma implies $x\approx y'$ for some cyclic conjugate $y'$ of $y$. However, by inspection of the diagram in the lemma we see that $y'$ has the same $q$ letters in the same order as $x$, and that the subwords between the $q$-letters are the same as those of $y$, hence $y'=_{M(G)}y=q^{\sigma_1}v_{1+j}q^{\sigma_2}...q^{\sigma_k}v_{k+j}$ for some $j$. Inspection of the diagram witnessing $c'(x,y')$ shows $$c(x,y)\leq c'(x,y')+\sum_{i}||v_i||+k\leq c'(x,y')+n.$$ Also, since $\sum_{i}||u_i||+||v_i||=\sum_{i}||u_i||+||v_{i+j}||,$ we have $$c'(x,y')\leq D_{k,\sigma}(\sum_{i}||u_i|+|v_i||)\leq D_{k,\sigma}(n).$$ Taking the maximum over all $x$ and $y$ of the above form gives $$D_{k,\sigma}(n)\leq D_{k,\sigma}'(n)+n.$$ An identical argument in the case of $qu\sim_{M(G)} q$ gives the upper bound $$D_0(n)\leq D_0'(n)$$ (note that, since there is only one $q$ corridor, $y'$ can be taken as just $y$). 

 For the lower bound, the previous lemma gives $c'(x,y')\leq tc(x,y)$. This implies $$\frac{D_{k,\sigma}'(n)}{t}\leq D_{k,\sigma}(n)$$ for all $n$, and similarly $$\frac{D_{0}'(n)}{t}\leq D_0(n).$$ 
\end{proof}

With these lemmas in hand, we can prove Theorem B.
\begin{proof}[Proof of Theorem B]
Suppose $q\alpha\approx q$ with $\alpha\in\langle X\rangle_{M(G)}$ and $||q\alpha||+||q||=n.$ Then $\alpha=_G1,$ that is, $\alpha\in\langle\langle R\rangle\rangle.$ Thus, $\lambda(\alpha)$ is defined. By the direct product structure of $H,$ every conjugator $\gamma\in H$ taking $q\alpha$ to $q$ via conjugation can be written in the the form $w\tau,$ where $w\in\langle X\rangle_{M(G)},\tau\in\langle\Theta\rangle_{M(G)},$ and $|\gamma|=~||w||+~||\tau||.$ Without loss of generality, assume $|\gamma|=c'(q\alpha,q)$ (such a word exists by the definition of $c'$). Then we have $$(w\tau)q\alpha(w\tau)^{-1}=_{M(G)}wu^{-1}qv\alpha w^{-1},$$ where $u,v$ are elements of $\langle X\rangle_{M(G)}$ with $v$ equal to $u$ with elements of $R^{\pm 1}$ inserted. Since $\langle X,q\rangle_{M(G)}$ is a free group, and $q$ appears here only once, we have that $w\equiv u$ and $v\alpha w^{-1}\equiv1.$ This implies $\alpha\equiv v^{-1}w\equiv v^{-1}u.$

Now, since $v$ is equal to $u$ with elements of $R^{\pm 1}$ inserted arbitrarily, we can write $v$ in the form $v_1r_1v_2r_2...v_mr_mv_{m+1}$ and $u$ in the form $v_1v_2...v_{m+1},$ where $v_1,...,v_{m+1}$ are arbitrary words, and $r_1,...,r_m\in R^{\pm1}.$ This means that $v^{-1}u$ can be written as  $$v^{-1}u\equiv v_{m+1}^{-1}r_m^{-1}v_m^{-1}...r_2^{-1}v_2^{-1}r_1^{-1}v_1^{-1}v_1v_2...v_{m+1}$$ $$\equiv v_{m+1}^{-1}r_m^{-1}v_m^{-1}...r_2^{-1}v_2^{-1}r_1^{-1}v_2...v_{m+1}.$$ We have a nested sequence of conjugations by $v_2^{-1},v_3^{-1},...,v_{m+1}^{-1},$ so this is freely equal to a word of the form $$\prod_{i=1}^mw_ir_{m-i+1}w_i^{-1},$$ where $w_1\equiv v_{m+1}^{-1},w_m\equiv v_2...v_{m+1},$ and $w_i^{-1}w_{i+1}\equiv v_{m-i+1}^{-1}.$ Thus $$f(\prod_{i=1}^mw_ir_{m-i+1}w_i^{-1})=m+\sum_{i=1}^{m-1}||v_{m-i+1}^{-1}||+||v_{m+1}^{-1}||+||v_2...v_{m+1}||.$$ The sum of first two terms is precisely $||\tau||,$ and the sum of the last two is  at most $2||u||,$ which is less than $2||\tau||$ by Lemma \ref{conjbythetas}, so $$f(\prod_{i=1}^mw_ir_{m-i+1}w_i^{-1})\leq 3||\tau||.$$ But, $$f(\prod_{i=1}^mw_ir_{m-i+1}w_i^{-1})\geq \lambda(\alpha)$$ by definition, so $3c'(qa,q)=3|\gamma|\geq 3||\tau||\geq \lambda(\alpha).$ Taking the maximum over all $\alpha$ with $\alpha=_G1$ and $||\alpha||\leq n-2,$ then dividing by 3, gives $D_0'(n)\geq \Lambda(n-2)/3$. This implies the desired lower bound.

Next, we establish an upper bound of $\Lambda(n-2).$ Let $\alpha=_G1$ be given, so $q\alpha\approx q.$ There exists a word $w_\alpha$ of the form $$\prod_{i=1}^mw_ir_{i}w_i^{-1}\equiv \alpha$$ with $f(w_\alpha)=\lambda(\alpha)\leq \Lambda(n-2).$ By the proof of the reverse implication of Lemma \ref{wordstoconjbythetas}, there exists a word $\gamma$ such that $\gamma q\gamma^{-1}=_{M(G)}qw_\alpha^{-1}$, and such that $|\gamma|=~f(w_\alpha^{-1})=~f(w_\alpha)=~\lambda(\alpha)$. Hence $c'(q\alpha,q)\leq \lambda(\alpha)\leq \Lambda(n-2).$ Taking the maximum again over all $\alpha$ with $\alpha=_G1$ and $|\alpha|\leq n-2$ gives $D_0'(n)\leq \Lambda(n-2)$. This implies the desired upper bound, so we are done.

\end{proof}

\section{Proof of Theorem C}
Now we turn to proving Theorem C. By Lemma \ref{DandDprimerelationship}, it is sufficient to bound $D_{k,\sigma}'$.

\begin{proposition}\label{Dprimelower} Fix $k>0$ and $\sigma=(\sigma_1,...,\sigma_k)$, with $\sigma_i\in \{\pm1\}$.
    
 For all $n>2k\geq 2,$
    $$D_{k,\sigma}'(n)\geq \textbf{C}_{k,\sigma}(n-2k) .$$
\end{proposition}
\begin{proof}
    Since $ \textbf{C}_{k}(n-2k)$ is defined as a maximum over a finite set, it equals $c_k(\textbf{u},\textbf{v})$ for some $\textbf{u}=(u_1,...,u_k)\overset{\sigma}{\sim} (v_1,...,v_k)=\textbf{v}$ with $$\sum_{j=1}^k||u_j||+||v_j||\leq n-2k.$$ Define $x=q^{\sigma_1}u_1q^{\sigma_2}...q^{\sigma_k}u_k$ and $y=q^{\sigma_1}v_1q^{\sigma_2}...q^{\sigma_k}v_k$. Theorem \ref{conjwithoutthetas} gives $x\approx y$. Also, $$||x||+||y||\leq 2k+\sum_{j=1}^k||u_j||+||v_j||\leq 2k+n-2k=n.$$ Let $\gamma\in \langle X\cup\Theta\rangle_{M(G)}$ be any conjugator so that $\gamma x\gamma^{-1}=_{M(G)}y$ and $|\gamma|\leq D'_{k,\sigma}(n).$ This exists because $\sum_{j}||u_j||+||v_j||\leq n-2k$ and $ c'(x,y)\leq D'_{k,\sigma}(n)$ by definition.
    
 Our argument here follows the proof of the forward direction of Theorem \ref{conjwithoutthetas}. We know $\gamma=z\tau$ for $z\in\langle X\rangle_{M(G)},\tau\in\langle \Theta\rangle_{M(G)}$. Let $w,\varepsilon$ be words on $X$ such that $\varepsilon=_G1$ and $\tau q\tau^{-1}=_{M(G)}w^{-1}qw\varepsilon$. By Lemma \ref{conjbythetas}, $||w||\leq~||\tau||\leq~|\gamma|$. By the argument in the proof of Theorem B, $w$ and $\varepsilon$ are words that make $(u_1,...,u_k)\overset{\sigma}{\sim} (v_1,...,v_k)$. By the definition of $\textbf{C}_{k,\sigma}$, then, $|w|\geq \textbf{C}_{k,\sigma}(n)$. Combining all the inequalities above gives the desired result.
\end{proof}

\begin{remark}
    For the special case of $k=1$ and $\sigma=(1)$, the above inequality becomes $D_{1,(1)}'(n)\geq~\Gamma_G(n-2)$, due to the existence of a uniform $\varepsilon$ becoming redundant.
\end{remark}

Next we show the upper bound to $D'_{k,\sigma}(n)$.

\begin{proposition}\label{Dprimeupper} Let $\sigma$, $t$, and $D_{k,\sigma}$ be as above, and suppose $\sigma$ is non-alternating. For all $n>2k$,
$$D'_{k,\sigma}(n)\leq (5M+1)\Lambda(2\textbf{C}_{k,\sigma}(n)+n)$$ for some constant $M\geq 1$ which depends only on $X$ and $R.$
\end{proposition} 
\begin{proof} Let $a=q^{\sigma_1}u_1q^{\sigma_2}...q^{\sigma_k}u_k\approx b=q^{\sigma_1}v_1q^{\sigma_2}...q^{\sigma_k}v_k$ be given, where $u_i,v_i$ are words over $X$ and $\sum_{i=1}^k||u_i||+||v_i||\leq n-2k$. Let $\gamma\in\langle X\cup\Theta\rangle_{M(G)} $ be the conjugator constructed in backwards direction of the proof of Theorem \ref{conjwithoutthetas}, and observe $c'_{M(G)}(a,b)\leq |\gamma|.$ Lastly, let  $\textbf{u}=(u_1,...,u_k),\textbf{v}=(v_1,...,v_k)$.

Recall from the proof of Theorem \ref{conjwithoutthetas} that $\gamma$ is written as $w\tau_{\varepsilon}\tau_w$ if $\sigma_1=1$ and as $w\varepsilon \tau_w\tau_\varepsilon$ if $\sigma_1=-1$, where $w$ and $\varepsilon$ are some words which make $\textbf{u}\overset{\sigma}{\sim}\textbf{v}$. By the construction of $\tau_\varepsilon$ in the proof of Lemma \ref{wordstoconjbythetas}, we see that $||\tau_\varepsilon||\leq \Lambda(||\varepsilon||)$. Also, since $\sigma$ is non-alternating, there exist some $\sigma_i$ such that $\sigma_i=\sigma_{i+1}$. If $\sigma_i=1$, then the definition of $\textbf{u}\overset{\sigma}{\sim}\textbf{v}$ gives that $v_i\equiv w\varepsilon u_iw^{-1}$, and hence $\varepsilon\equiv w^{-1}v_iwu_i^{-1}$. This means $||\varepsilon||=||wu_iw^{-1}v_i^{-1}||$. If $\sigma_i=-1$, then $v_i\equiv wu_i\varepsilon^{-1}w^{-1}$ which implies $\varepsilon^{-1}\equiv w^{-1}u_i^{-1}v_iw$, so $||\varepsilon||=||\varepsilon^{-1}||=||w^{-1}u_i^{-1}v_iw||$. In both cases, we see $$||\varepsilon||\leq 2||w||+\max_{i}(||u_i||+||v_i||)\leq 2||w||+n.$$  Similarly, $||\tau_w||=||w||$, and taking $w$ to be the smallest element such that $wu_iwv_i^{-1}\equiv \varepsilon$ for all $i$ gives $||w||=c_{k,\sigma}(\textbf{u},\textbf{v})\leq \textbf{C}_{\sigma}k(n)$. The words $w\tau_\varepsilon\tau_w$ and $w\varepsilon \tau_w\tau_\epsilon$ are reduced the generators of $\langle X\cup\Theta\rangle_{M(G)}$, so $$|\gamma|\leq ||w||+||\epsilon||+||\tau_w||+||\tau_\varepsilon||\leq ||w||+(2||w||+n)+||w||+\Lambda(2||w||+n)$$ $$\leq 4||w||+n+\Lambda(2||w||+n)$$ $$\leq 4\textbf{C}_{k,\sigma}(n)+n+\Lambda(2\textbf{C}_{k,\sigma}(n)+n).$$ Using our assumption that the set of relators $R$ is non-empty, it is an easy exercise to show that $\Delta(n)$, and thus (by Proposition \ref{lambdadehnrel}) $\Lambda(n)$, is bounded below by $n/M$ for some constant $M$ depending on $X$ and $R.$ Since $\Lambda(n)$ is non-decreasing we can therefore condense this expression by 
\begin{align*} 
& 4\textbf{C}_{k,  \sigma}(n)+n+\Lambda(2\textbf{C}_{k,\sigma}(n)+n)   \\  
& \leq   \ 4M\Lambda(\textbf{C}_{k,\sigma}(n))+n+\Lambda(2\textbf{C}_{k,\sigma}(n)+n) \\ 
& \leq \ 4M\Lambda(2\textbf{C}_{k,\sigma}(n)+n)+n+\Lambda(2\textbf{C}_{k,\sigma}(n)+n) \\  
& \leq \  (4M+1)\Lambda(2\textbf{C}_{k,\sigma}(n)+n)+n\\
& \leq \  (4M+1)\Lambda(2\textbf{C}_{k,\sigma}(n)+n)+M\Lambda(n)\\
& \leq \  (4M+1)\Lambda(2\textbf{C}_{k,\sigma}(n)+n)+M\Lambda(2\textbf{C}_{k,\sigma}(n)+n)\\
& \leq \ (5M+1)\Lambda(2\textbf{C}_{k,\sigma}(n)+n).
\end{align*}
This completes our proof.
\end{proof}

These two propositions, combined with Lemma \ref{DandDprimerelationship}, prove Theorem C.

\section{Direct implication of Theorem A}\label{forwardtheoremA}
Now we turn to a proof of the forward direction for Theorem A. Suppose $G$ is such that the conjugacy problem for $M(G)$ is solvable. We first give an algorithm for deciding $\textbf{u}\overset{\sigma}{\sim} \textbf{v}$ whenever $\sigma$ is non-alternating.

\begin{lemma}\label{solvesnonalternating}
    Suppose the word problem for $G$ is solvable. If $\textbf{C}_{k,\sigma}$ is a computable function for all $k>0$ and non-alternating $\sigma=(\sigma_1,...,\sigma_k)$, with $\sigma_i\in~\{\pm1\}$, then there is an algorithm deciding $\textbf{u}\overset{\sigma}{\sim}\textbf{v}$ for inputs $\textbf{u},\textbf{v}, \sigma$, where $\textbf{u}=(u_1,...,u_k)$ and $\textbf{v}=(v_1,...,v_k)$ are $k$-tuples on $F(X)$ and $\sigma=(\sigma_1,...,\sigma_k)$ is a $k$-tuple on $\{\pm1\}$, such that $\sigma$ is non-alternating and $\sigma_i=-\sigma_{i+1}$ implies $u_i$ and $v_i$ are non-trivial.
\end{lemma}

\begin{proof}
    Let $\textbf{u},\textbf{v}$, and $\sigma$ be given as above, and let $n=\sum_i||u_i|+||v_i||$. By definition, $\textbf{u}\overset{\sigma}{\sim}\textbf{v}$ if and only if there exist $w$ and $\varepsilon$ such that they satisfythe definition of $\overset{\sigma}{\sim}$ and $||w||\leq \textbf{C}_{k,\sigma}(n)$. Also, for any given $w',\varepsilon'$, it is decidable whether they make $\textbf{u}\overset{\sigma}{\sim}\textbf{v}$, since we can use the word problem of $G$ to check $\varepsilon'=_G1$, and the rest of the equalities in the definition are all in the free group on $X$, where the word problem is decidable. 
    
    Now, since $\sigma$ is non-alternating, there exists some $\sigma_j$ such that $\sigma_j=\sigma_{j+1}$. In this case, $w$ and $\varepsilon$ witnessing $\textbf{u}\overset{\sigma}{\sim}\textbf{v}$ implies either $w\varepsilon u_i w^{-1}\equiv v_i$ (if $\sigma_j=1$) or $wu_i(w\varepsilon )^{-1}\equiv v_i $ (if $\sigma_j=-1$). Either way, $\varepsilon$ can be computed directly by $w$. Thus, we can decide if $\textbf{u}\overset{\sigma}{\sim}\textbf{v}$ by iterating over all words $w$ with $||w||\leq\textbf{C}_{k,\sigma}(n)$ (which is computable by assumption), then computing $\varepsilon$ according to the above equations, and finally checking whether they witness $\textbf{u}\overset{\sigma}{\sim}\textbf{v}$. If no such $w$ is found, we know $\textbf{u}\overset{\sigma}{\not\sim}\textbf{v}$, and we are done.
\end{proof}

Now, it follows from Theorem B (also Lemma \ref{Sapirlemma}) that the conjugacy problem for $M(G)$ solves the word problem for $G$. Also, the conjugacy problem for $M(G)$ is solvable if and only if the conjugator length function $\Gamma_{M(G)}$ is computable. By definition, $\Gamma_G(n)\geq D_{k,\sigma}(n-2k)$ for all $k>0$ and all (possibly alternating) $\sigma=(\sigma_1,...,\sigma_k)$, with $\sigma_i\in \{\pm1\}$. Thus, if the conjugacy problem for $M(G)$ is solvable, $D_{k,\sigma}$ is a computable function, hence $\textbf{C}_{k,\sigma}$ is as well by Theorem C. By the above lemma, this gives that, if the conjugacy problem for $M(G)$ is solvable, there is an algorithm deciding $\textbf{u}\overset{\sigma}{\sim}\textbf{v}$ for all non-alternating $\sigma$.

Next, we give an algorithm for the alternating case. First, recall the following facts about conjugacy in free groups.
\paragraph{Fact 1} Let $a,b,\gamma,\gamma'$ be elements of some free group $F$. Suppose $\gamma a \gamma^{-1}\equiv b$. We have $\gamma' a \gamma'^{-1}\equiv b$ if and only if $\gamma'\equiv \gamma a^m$ for some $m\in \ZZ$.
\paragraph{Fact 2} Since every free group is torsion free and hyperbolic, by the work of \cite{BridsonHowie}, their list conjugacy problem is solvable. That is, given two tuples $(a_1,...,a_m)$ and $(b_1,...,b_m)$, we can compute whether there exists an $s$ such that $sa_is^{-1}\equiv b_i$ for $i=1,...,m$. If there exists such an $s$, it can be computed directly by iterating through elements of the given free group.
\paragraph{Fact 3} Since free groups are coherent Right-Angled Artin Groups, by \cite[Corollary 1.3]{Kapovich2003FoldingsGO} we see that the Cyclic Subgroup Membership problem is solvable for every free group (indeed, the general Subgroup Membership problem is). That is, given $a$ and $b$ in the free group, we can decide whether there exists $m\in \ZZ$ such that $b\equiv a^m$. If so, we can compute this $m$ by iterating through $\ZZ$.

With this fact in mind, we now prove two useful lemmas.
\begin{lemma}\label{conjsolvescyclic}
    Suppose the conjugacy problem for $M(G)$ is solvable. Then the Cyclic Subgroup Membership problem for $G$ is solvable. 
\end{lemma}
\begin{proof}
    Let $a$ and $b$ be reduced words representing elements of $G$. We wish to determine whether $b\in \langle a\rangle_G$, that is, $b=_G a^m$ for some $m\in \ZZ$. If $b=_G1$ then $b=_Ga^0$ automatically, and if $a=_G1$ then this occurs if and only if $b=_G1$ as well. Both cases can be checked using a solution to the word problem for $G$, which we remarked above must be solvable if the conjugacy problem for $M(G)$ is solvable. Thus, suppose $a\neq_G1$ and $b\neq_G1$. Then in particular $a\not\equiv 1$ and $b\not\equiv 1$. Since $a$ and $b$ are reduced and non-trivial, $bab^{-1}\not\equiv 1$. Thus, if $qaq^{-1}a\sim_{M(G)}qaq^{-1}bab^{-1}$ then, by inspection of the witnessing diagram, $qaq^{-1}a\approx qaq^{-1}bab^{-1}$. Let $\gamma$ be the freely reduced word equal to $bab^{-1}$. By Theorem \ref{conjwithoutthetas}, $qaq^{-1}a\approx qaq^{-1}bab^{-1}$ implies $(a,a)\overset{(1,-1)}{\sim}(a,\gamma)$. That is, there exists some $w$ and $\varepsilon$ such that $\varepsilon=_G1$, $w\varepsilon a (w\varepsilon)^{-1}\equiv a$, and $waw^{-1}\equiv \gamma\equiv bab^{-1}$. Since $bab^{-1}\equiv bab^{-1}$ automatically, we have $w=ba^{m_1}$ for some $m_1\in \ZZ$. Likewise, we have $w\varepsilon\equiv a^{m_2}$ for some $m_2\in \ZZ$. Combining these equations gives $ba^{m_1}\varepsilon\equiv a^{m_2}$. Since $\varepsilon=_G1$, this means $ba^{m_1}=_Ga^{m_2}$, that is, $b=a^{m_2-m_1}$. This entire proof consists of a sequence of biconditionals, hence $b\in \langle a\rangle_{G}$ if and only if $qaq^{-1}a\approx qaq^{-1}bab^{-1}$, which we can check using the solution to the conjugacy problem for $M(G)$, so we are done.
\end{proof}

\begin{lemma}\label{solvealternating}
    Suppose the word problem for $G$ is solvable. If $\textbf{C}_{k,\sigma}$ is a computable function for all even $k>0$ and alternating $\sigma=(\sigma_1,...,\sigma_k)$, with $\sigma_i\in \{\pm1\}$, then there is an algorithm deciding $\textbf{u}\overset{\sigma}{\sim}\textbf{v}$ for inputs $\textbf{u},\textbf{v}, \sigma$, where $\textbf{u}=(u_1,...,u_k)$ and $\textbf{v}=(v_1,...,v_k)$ are $k$-tuples on $F(X)$ and $\sigma=(\sigma_1,...,\sigma_k)$ is a $k$-tuple on $\{\pm1\}$, such that $\sigma$ is alternating and $\sigma_i=-\sigma_{i+1}$ implies $u_i$ and $v_i$ are non-trivial.
\end{lemma}
\begin{proof}
    Without loss of generality, suppose $\sigma=(-1,1,-1,...,1)$. Also, let $n=\sum_i||u_i||+||v_i||$. We see that $\textbf{u}\overset{\sigma}{\sim}\textbf{v}$ if and only if there exist $w$ and $\varepsilon$ such that $\varepsilon=_G1$, $||w||\leq \textbf{C}_{k,\sigma}(n)$, and the equations $$wu_{2j+1}w^{-1}\equiv v_{2j+1}$$ $$w\varepsilon u_{2j}(w\varepsilon)^{-1}\equiv v_{2j}$$ hold for all $j=0,...,k/2$. Thus, by Fact 2, if $\textbf{u}\overset{\sigma}{\sim}\textbf{v}$ then there exist computable elements $r,s$ of $F(X)$ such that the equations $$ru_{2j+1}r^{-1}\equiv v_{2j+1}$$ $$s u_{2j}s^{-1}\equiv v_{2j}$$ hold $-$ if there does not exist such $r$ and $s$ we may conclude $\textbf{u}\overset{\sigma}{\not \sim}\textbf{v}$. If $s$ and $r$ exist, then by Fact 1 we have $$wu_{2j+1}w^{-1}\equiv v_{2j+1}$$$$w\varepsilon u_{2j}(w\varepsilon)^{-1}\equiv v_{2j}$$ if and only if there exist $m_i\in \ZZ$ ($i=1,...,k$) such that both $$w\equiv ru_{2j+1}^{m_{2j+1}}$$ $$w\varepsilon\equiv sv_{2j}^{m_{2j}}$$ hold for all $j=0,...,k/2$. Note that $r$ and $s$ can be computed independently of $w$ and $\varepsilon$. 
    
    This implies $\textbf{u}\overset{\sigma}{\sim}\textbf{v}$ if and only if there exist $r$ and $s$ computable as above, and there exist $w$  such that \begin{enumerate} \item $||w||\leq \textbf{C}_{k,\sigma}(n)$, and \item there exist $\varepsilon=_G1$ and $m_i\in \ZZ$ ($i=1,...,k$) such that $$w\equiv ru_{2j+1}^{m_{2j+1}}$$ $$w\varepsilon\equiv sv_{2j}^{m_{2j}}$$ hold for all $j=0,...,k/2$.\end{enumerate} The equation $w\varepsilon\equiv sv_{2j}^{m_{2j}}$ can be rewritten as $s^{-1}w\varepsilon\equiv v_{2j}^{m_{2j}}$, hence item (2.)\ above is equivalent to the fact that there exists a $\varepsilon=_G1$ and $m_i\in \ZZ$ such that $$w\equiv ru_{2j+1}^{m_{2j+1}}$$ $$s^{-1}w\varepsilon\in \bigcap_{j=1}^{k/2}\langle v_{2j}\rangle$$ for $j=1,...,k/2$, where here the subgroups $\langle v_{2j}\rangle$ are in the free group $F(X)$. Note that $$\bigcap_{j=1}^{k/2}\langle v_{2j}\rangle$$ is the intersection of cyclic subgroups of a free group, hence it is cyclic. We can compute a generating set $S$ by \cite{AvenhausMadlener}, and then compute a single generator $g$ from that set by iterating through elements of $F(X)$ and finding an element $g$ such that $g\in F(S)$ and $S\subseteq \langle g\rangle$ (if we reach an element with longer reduced-length than any element of $S$, we may conclude our intersection is trivial and take $g=1$). Thus $\bigcap_{j=1}^{k/2}\langle v_{2j}\rangle=\langle g\rangle$ for some $g$ computable from $v_2,...,v_{k/2}$. Since we are assuming $\varepsilon=_G1$, this is in turn equivalent to saying $$s^{-1}w\in \bigcap_{j=1}^{k/2}\langle v_{2j}\rangle_G,$$ hence our fact (and therefore (2.))\ is equivalent to the following statement: there exist $m_i\in \ZZ$ such that $w\equiv ru_{2j+1}^{m_{2j+1}}$ and $s^{-1}w\in \langle g\rangle_G$ for $j=1,...,k/2$, where now the subgroups $\langle v_{2j}\rangle$ are subgroups of $G$.

    In light of this argument, our algorithm proceeds as follows. First, compute $r$ and $s$. Then, for every $w$ with $||w||\leq \textbf{C}_{k,\sigma}(n)$, check whether ($2''$) holds using the solution to the Cyclic Subgroup Membership problems for both $F(X)$ (to see if $w\equiv ru_{2j+1}^{m_{2j+1}}$ for some $m_{2j+1}\in \ZZ$) and $G$ (to see if $s^{-1}w\in \langle g\rangle_G$).
\end{proof}

Now, if the conjugacy problem for $M(G)$ is solvable, the above remarks show that the hypotheses to Lemma \ref{solvealternating} are satisfied. Thus, combining this Lemma with Lemma \ref{solvesnonalternating} gives a proof of the forward direction of Theorem A.

\printbibliography
\end{document}